%% file: main.tex
\documentclass[12pt]{article}

\usepackage[margin=1in]{geometry}
\usepackage{tikz}
\usepackage{xcolor} 
\usepackage{multirow}
\usepackage{amsmath,amssymb,amsthm}
\usepackage[colorlinks=true, allcolors=blue]{hyperref}
\usepackage[capitalize,nameinlink]{cleveref}
\usepackage{float}

\usepackage{dsfont} 
\usetikzlibrary{decorations.pathreplacing} 

\newtheorem{theorem}{Theorem}[section]
\newtheorem{proposition}[theorem]{Proposition}
\newtheorem{corollary}[theorem]{Corollary}
\newtheorem{lemma}[theorem]{Lemma}

\theoremstyle{definition}

\newtheorem{conjecture}{Conjecture}[section]
\newtheorem{definition}{Definition}[section]
\newtheorem{example}{Example}[section]

\input{paper-pics}

\usepackage{mathtools}

\DeclareMathOperator{\FR}{FR}

\usepackage[backend=biber, style=alphabetic, sorting=ynt]{biblatex}
\title{Enumerating Prime Patterns in Juggling Variations}
\author{Steve Butler\thanks{Iowa State University, Ames, IA 50011, USA. \href{mailto:butler@iastate.edu}{\texttt{butler@iastate.edu}}} 
\and Vera Choi\thanks{Tufts University, Medford, MA 02155, USA. \texttt{\href{mailto:vera.choi@tufts.edu}{vera.choi@tufts.edu}}}
\and Joel Jeffries\thanks{Iowa State University, Ames, IA 50011, USA. \href{mailto:joel.a.jeffries@gmail.com}{\texttt{joel.a.jeffries@gmail.com}}} 
\and Nina McCambridge\thanks{Carnegie Mellon University, Pittsburgh, PA 15213, USA. \href{mailto:nmccambr@andrew.cmu.edu}{\texttt{nmccambr@andrew.cmu.edu}}}
\and Asia Morgenstern \thanks{Westminster College, New Wilmington, PA 16172, USA. Current address: University of Kentucky, Lexington, KY 40506, USA. \href{mailto:asia.morgenstern@gmail.com}{\texttt{asia.morgenstern@gmail.com}}} 
\and Samuel Orellana Mateo\thanks{Duke University, Durham, NC 27708, USA. \href{mailto:samuelorellanamateo@gmail.com}{\texttt{samuelorellanamateo@gmail.com}}}}
\date{\empty}

\begin{document}

\maketitle

\begin{abstract}
    Juggling patterns can be mathematically modeled as closed walks within directed state graphs. In this paper, we present a unified framework of unbounded juggling patterns and its variations (including multiplex, colored, and passing) primarily through the formalism of the juggling state. By extending this state-based approach and utilizing combinatorial tools such as set partitions and filled Ferrers diagrams, we find and prove a new lower bound on the number of $b$-ball prime patterns with period $n$. Further, we determine exact counts for 2-ball multiplex, 1-ball passing, and 2-ball colored juggling patterns, as well as a lower bound for 2-ball passing. We also provide an extensive analysis of the asymptotic growth rates for these pattern counts. Finally, we formalize the infinite state graph, $G_\infty$, and utilize flip-reverse involutions to establish bijections between classes of prime patterns, exploring how fixing a specific state influences the enumeration of prime walks.
\end{abstract}


\newpage

\section{Introduction} \label{sec:introduction}

\subsection{State Graph}

Juggling, an art form celebrated for its physical dexterity, also has a rich mathematical structure. The connections between juggling and combinatorics are particularly strong, with the state graph as a formalism developed to describe and analyze juggling patterns. 

In this model, a juggling state is represented by a binary vector indicating the future landing times of the balls. For instance, in a $b$-ball pattern, a state is a vector with exactly $b$ ones. A $1$ in position $i$ signifies that a ball is scheduled to land $i$ time beats from the present.

A directed edge connects two states if one can legally transition to the other in a single beat. If the first entry is a $0$, then no ball will land in the next time beat, hence achieving the only possible transition by deleting that entry. If the first entry is a $1$, then a ball lands; we delete the first entry and we replace any $0$ by a $1$, representing a throw of the ball that just landed.

While Buhler et al. \cite{Buhler1994} formally model juggling patterns as permutations of integers $f: \mathbb{Z} \to \mathbb{Z}$ where $f(t) \ge t$, we rely entirely on the state graph formalism. 

Formally, given a number of balls $b$, if $\alpha = \langle \alpha_1, \alpha_2, \dots \rangle$ and $\beta = \langle \beta_1, \beta_2, \dots \rangle$ are two states where $\alpha_i, \beta_j \in \{0, 1\}$ for all $i, j$, then $\alpha \to \beta$ is a possible transition if and only if $\alpha_{i+1} \leq \beta_i$ for all $i = 1, 2, \dots$ and
\[
\sum_{i=1}^\infty \alpha_i = \sum_{i=1}^\infty \beta_i = b
\]

A periodic juggling pattern of length $n$ corresponds to a closed walk of length $n$ in this infinite state graph. This is a subgraph for $b = 2$ drawn by \cite{Banaian2016}:

\begin{figure}[ht]
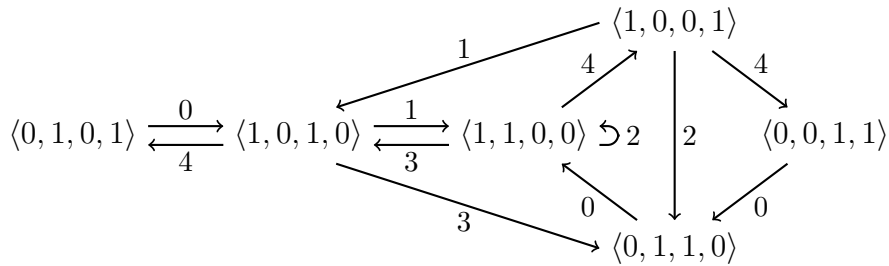

\begin{center}
\introstategraph
\caption{A subgraph of the two-ball state graph. The edge labels correspond to throw heights.}
\label{fig:state}
\end{center}
\end{figure}

While the total number of juggling patterns of a given length is well-understood, a more challenging  problem is to enumerate the \textbf{prime juggling patterns}. A pattern is defined as prime if its corresponding closed walk in the state graph is a cycle, meaning it visits $n$ distinct states before returning to the start. These prime patterns are fundamental, as any juggling pattern can be decomposed into a sequence of them. This concept is identical to what Polster \cite{Polster2003} refers to as ``prime loops'' or what Banian et al. \cite{Banaian2016} define as ``prime cycles''. While Polster mainly bounds state graphs by a maximum throw height $h$, our definitions allow for unbounded heights. Furthermore, this differentiates them from the ``primitive'' sequences enumerated by Butler and Graham \cite{Butler2010}, which only require the starting state to not be revisited.

\subsection{Work on prime patterns and formalization}

Significant progress on this problem was made in the paper \textit{Counting prime juggling patterns} by Banaian, Butler, et al. \cite{Banaian2016}, which is a foundational reference for our work. For the case of two balls ($b=2$), they established a bijection between prime patterns of length $n$ and ordered collections of sets of ``spacings'' (the distance between the two balls in the air). The largest elements of these sets form a partition of $n$ into distinct parts. This connection allowed them to derive a formula for $P'(2,n)$, the number of two-ball prime patterns of length $n$.

This paper looks at prime patterns by defining a bijection between them and series of sets with certain conditions, one being that it is cyclic (such that $S_1, S_2, S_3$ represents the same pattern as $S_2, S_3, S_1$). The proposition they use, which they don't explicitly state, but lies on their argument for the main result, is the following:

\begin{lemma}\label{lem:pattern-sets-bijection}
    Let $\mathcal{X}_{t,n}$ the set of all $X_{t,n} = (S_1, \dots, S_t)$ (up to cyclic order), where
    \[
        n = \sum_{i=1}^t \max(S_i)
    \]
    and $S_i \cap S_j = \varnothing$ for all $1 \leq i < j \leq t$. This is possibly empty if $t$ is not large enough.

    Then, there is a bijection between the set of prime patterns $\rho$ of length $n$ that contain $t$ $C_2$ cards in their construction, and elements in $\mathcal{X}_{t,n}$.
\end{lemma}

A key result from \cite{Banaian2016} that we will use extensively is the function $c_t(n)$, which counts the number of ways to form valid, prime-generating spacing sets where the largest spacings come from a partition of $n$ into $t$ distinct parts. Essentially:

\begin{definition}\label{def:c_t(n)}
\[
    c_t(n) = \vert \mathcal{X}_{t,n} \vert
\]
\end{definition}

Then, by using Ferrers diagram and a counting argument, \cite{Banaian2016} shows that:

\begin{proposition}   
\[
c_t(n) = \sum_{\substack{p_1>\cdots>p_t \ge 1 \\ p_1+\cdots+p_t=n}}
\frac{1}{t(t+1)}\prod_{i=1}^t\bigg(\frac{i+1}{i}\bigg)^{p_i}
\]
\end{proposition}

An immediate consequence of this is that the total number of normal prime patterns is $N'(2,n) = \sum_t c_t(n)$. Asymptotically, they showed that $N'(2,n) \sim \gamma \cdot 2^n$ for a constant $\gamma \approx 1.3296$. While Banaian et al. utilized the $c_t(n)$ function strictly for normal 2-ball patterns, we demonstrate the versatility of this function in later sections by using it as a building block for other counts.

\subsection{Variations of Juggling}

The work of Buhler, Eisenbud, Graham, and Wright, ``Juggling Drops and Descents'' \cite{Buhler1994}, laid the mathematical foundation for the modern analysis of juggling patterns. They established many of the core concepts, formalizing juggling patterns as permutations of the integers \cite{Buhler1994} and providing the condition for a sequence of throws to constitute a valid pattern \cite{Buhler1994}. A key contribution of their work was the introduction of the function $N(b, n)$, which counts the number of periodic juggling patterns with $n$ beats and $b$ balls \cite{Buhler1994}. This function serves as the primary inspiration for our research, as we seek to define and analyze analogous counting functions for different juggling variations and for the more restrictive class of prime patterns.

The paper by Bass and Butler \cite{Bass2024} explores the enumeration of juggling passing patterns within a ``saturated'' model, where all $k$ jugglers are assumed to be active on every beat of the pattern. By counting this specific set of patterns in two distinct ways, the authors provide a combinatorial proof of a generalized Worpitzky’s identity. They use generalized Eulerian numbers and rook placements on a $(kn) \times n$ board, and while this counts all valid sequences, it does not focus on prime patterns.

We, in contrast, analyze a different ``sparse'' system by focusing on prime patterns with a fixed, small number of balls (one or two), a scenario where not all hands are necessarily active, finding an exact count for $P'(1, n, k)$ and a lower bound for $P'(2, n, k)$.

In a note added in proof to their 1994 paper, Buhler et al. \cite{Buhler1994} mentioned that the foundational work by Ehrenborg and Readdy \cite{EHRENBORG1996107} pioneered the enumeration of multiplex juggling patterns. They defined patterns as triples $(d, x, a)$ and utilized decks of cards indexed by multisets of crossings to derive their main $q$-analogue identities. However, when counting physically distinct prime patterns this abstraction, their reliance on ``internal crossings'' implies that a bijection between card sequences and unique juggling states cannot be established. Polster \cite{Polster2003} also explored multiplex state graphs, relying on a bracket notation (for example, $[11]10$) to group balls landing simultaneously. Hence, we redefine these concepts.

A more rigorous framework for multiplex juggling was later provided by Butler and Graham \cite{Butler2010}, who transformed the problem of counting sequences into one of filling a binary matrix subject to specific constraints. While their work successfully enumerated periodic and primitive sequences, they explicitly posed the enumeration of \textbf{prime} multiplex sequences as an open problem. Building on this foundation, we narrow our focus to the exact enumeration of these prime patterns, specifically for two balls. By transitioning from Polster's bracket notation to a concrete coordinate notation $\sigma_{(i,j)}$, and defining a new, unambiguous set of cards ($D_0, D_a, D_b, D_c$), we have a one-to-one correspondence between our card sequences and the prime patterns themselves. This isolates the internal crossings described by Ehrenborg and Readdy (captured by our $D_b$ card) and gives an exact count for $M'(2, n)$, thereby resolving Butler and Graham's open problem for the two-ball case.

\subsection{Infinite State Graph} \label{sec:intro-infinite-state-graph}

To study the broader implications of juggling, researchers have also explored infinite state models. Banaian et al. \cite{Banaian2016} observed that the state graph for $b$ balls is isomorphic to an induced subgraph of the state graph for $b+1$ balls. Furthermore, Knutson, Lam, and Speyer \cite{Knutson_Lam_Speyer_2013} established a connection between juggling patterns and algebraic geometry by formalizing ``virtual juggling states'' and ``juggling functions.'' In their framework, the negative integers act as a ``Dirac sea'' of available balls, and patterns are modeled as affine permutations $f: \mathbb{Z} \to \mathbb{Z}$ in the extended affine Weyl group $\tilde{A}_{n-1}$, where $f(t)$ denotes the landing time of a ball thrown at time $t$.

This infinite modeling allows for the analysis of properties independent of a finite number of balls. This has been relevant in modern combinatorics; for instance, Galashin and Lam \cite{Galashin2024}  used  bounded affine permutations and juggling crossing statistics to compute positroid Catalan numbers, linking juggling math to knot invariants and Khovanov-Rozansky homology. 

In \cref{sec:infinite-state-graph}, we build upon this legacy by formally defining the infinite state graph $G_\infty$. Notice that our abbreviated infinite binary sequences conceptually mirror the ``virtual juggling states'' of Knutson et al., allowing us to establish bijections between walk lengths by utilizing flip-reverse involutions across the infinite graph.

Our formalization of $G_\infty$ provides a direct answer to an open question posed at the conclusion of Buhler et al.'s appendix \cite{Buhler1994}: \textit{``How can these ideas be used to describe, or 'name,' juggling patterns with infinitely many balls?''}.

\subsection{Cards Notation}

Another way to represent juggling patterns is using their cards. A $b$-ball juggling pattern can be represented using $b+1$ cards, $C_0, C_1, \dots, C_b$, where $C_0$ means no ball landed, and $C_i$ means a ball landed and it was thrown to be now in relative position $i$. These cards are well studied, and a valid pattern represented by states has a known bijection with its card representation (see \cite{Buhler1994, Polster2003, Banaian2016}). It is also important to note that our use of cards differs from the definition used by Bass and Butler \cite{Bass2024}. Every beat in a juggling pattern has a corresponding card. 

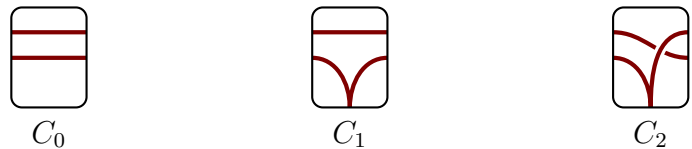
\begin{figure}[h]
\centering
\begin{tikzpicture}
\cardA{-4}0
\node at (-3.5,0) [below = 1] {$C_0$};
\cardB{0}0
\node at (0.5,0) [below = 1] {$C_1$};
\cardC{4}0
\node at (4.5,0) [below = 1] {$C_2$};
\end{tikzpicture}
\caption{Cards for $b = 2$}
\label{fig:juggling-cards}
\end{figure}

For example, in \cref{fig:state}, the pattern $\langle 0,1,1,0 \rangle \to \langle 1,1,0,0 \rangle \to \langle 1,0,0,1 \rangle \to \langle 0,1,1,0 \rangle$ of length $3$ can also be represented by the sequence $C_0 C_2 C_1$ or any rotation of these.

For example, in \cref{fig:state}, the pattern $\langle 0,1,1,0 \rangle \to \langle 1,1,0,0 \rangle \to \langle 1,0,0,1 \rangle \to \langle 0,1,1,0 \rangle$ of length $3$ can also be represented by the sequence $C_0 C_2 C_1$ or any rotation of these. Following the card sequences utilized by Banaian et al. to track relative ball order, we will introduce the $D$ cards to unambiguously track simultaneous landings and crossings in multiplex prime patterns.

\subsection{Our work}

We take the idea of using states to define patterns and extend it to its variations (multiplex, passing) and introduce a new variation: colored. In their concluding remarks, Butler and Graham \cite{Butler2010} speculated on the combinatorial complexity of enumerating walks where balls are distinct rather than identical, which our colored variation formalizes. We aim to count the number of prime patterns for each of these variations. To do this, we establish a new notation framework. While Banaian et al. \cite{Banaian2016} used $P(n,b)$ to denote the count of normal prime patterns, we redefine normal prime patterns as $N'(b,n)$ to align with the standard periodic pattern notation $N(b,n)$ from Buhler et al. \cite{Buhler1994}. This allows us to reserve $P$ for Passing patterns, $M$ for Multiplex, and $C$ for Colored. We use cursive to refer to the sets of the patterns and the standard Latin letter to refer to the count of the respective set.

\begin{table}[h!]
    \centering
    \begin{tabular}{|l|l|l|}
        \hline
        \textbf{State Representation} & \textbf{Constraint on Entries} & \textbf{Constraint on $k$} \\
        \hline \hline
        $\alpha = \langle \alpha_1, \alpha_2, \dots \rangle$ & $\alpha_i \in \{0, 1\}$ & N/A \\
        \hline
        $\alpha = \langle \alpha_1, \alpha_2, \dots \rangle$ & $\alpha_i \in \{0, 1, \dots, k\}$ & $k$apacity per hand $k \le b$ \\
        \hline
        $\alpha = \langle \alpha_1, \alpha_2, \dots \rangle$ & $\alpha_i \in \{ 0, \underbrace{\textcolor{orange}{1}, \textcolor{blue}{1}, \dots, \textcolor{violet}{1}}_{k} \}$ & number of $k$olors $k \le b$ \\
        \hline
        $A = (\alpha_{i,j})$, a $k \times \infty$ matrix  & $\sigma_{i,j} \in \{0, 1\}$ & number of hands $k \in \mathbb{Z}$ \\
        \hline
    \end{tabular}
    \caption{State representations and constraints for juggling variations.}
    \label{tab:state_representations}
\end{table}

Where $\sum_{i = 1}^{\infty} \alpha_i = b$, or $\sum_{i=1}^{k} \sum_{j=1}^{\infty} \alpha_{i,j} = b$ for the last case. The condition for a valid transition $\alpha \to \beta$ is $\alpha_{i+1} \leq \beta_i$ for $i = 1, 2, \dots$, and the invariant of the sum being $b$. Similarly, for passing, the condition for a valid transition $A \to B$ is $\alpha_{i, j+1} \leq \beta_{i, j}$ for $i = 1, 2, \dots, k$ and $j = 1, 2, \dots$, again respecting the invariant.

\begin{table}[H]
    \centering
    \begin{tabular}{|l|c|c|}
        \hline
        \textbf{Juggling Variation} & \textbf{Periodic Patterns} & \textbf{Prime Patterns} \\
        \hline \hline
        Normal & 
        \begin{tabular}{@{}c@{}} $\mathcal{N}(b, n)$ \\ $N(b, n)$ \end{tabular} & 
        \begin{tabular}{@{}c@{}} $\mathcal{N}'(b, n)$ \\ $N'(b, n)$ \end{tabular} \\
        \hline
        Multiplex & 
        \begin{tabular}{@{}c@{}} $\mathcal{M}(b, n, k)$ \\ $M(b, n, k)$ \end{tabular} &
        \begin{tabular}{@{}c@{}} $\mathcal{M}'(b, n, k)$ \\ $M'(b, n, k)$ \end{tabular} \\
        \hline
        Colored & 
        \begin{tabular}{@{}c@{}} $\mathcal{C}(b, n, k)$ \\ $C(b, n, k)$ \end{tabular} &
        \begin{tabular}{@{}c@{}} $\mathcal{C}'(b, n, k)$ \\ $C'(b, n, k)$ \end{tabular} \\
        \hline
        Passing & 
        \begin{tabular}{@{}c@{}} $\mathcal{P}(b, n, k)$ \\ $P(b, n, k)$ \end{tabular} &
        \begin{tabular}{@{}c@{}} $\mathcal{P}'(b, n, k)$ \\ $P'(b, n, k)$ \end{tabular} \\
        \hline
    \end{tabular}
    \caption{Notation for the sets (cursive) and counts (latin) of juggling patterns. Here, $b$ is the number of balls, $n$ is the period, $m$ is the multiplex capacity, and $k$ is the number of colors or hands.}
    \label{tab:notation_prime}
\end{table}

This paper builds upon the framework established above. A general formula for $N(b, n)$ was found using Mobius inversion by Buhler et al \cite{Buhler1994} in 1994. Then, in 2015, an exact count for $N'(2,n)$ was found by Banaian et al \cite{Banaian2016}.

We extend this work by providing an exact count for $M'(2, n) = M'(2, n, 2)$ in \cref{sec:multiplex-2-ball} and also an exact count for $C'(2, n) = C'(2, n, 2)$ in \cref{sec:colored-2-ball}, basing our argument on the $c_t(n)$ function developed by \cite{Banaian2016}. We also find an exact count for $P'(1, n, k)$ in \cref{sec:passing-1-ball} and a lower bound for $P'(2, n, k)$ in \cref{sec:passing-2-ball}. Additionally, we improve the currently standing lower bound for $N'(b, n)$, which was $\frac{1}{b} b^n$, presented in \cite{Banaian2016}. 

Finally, \cref{sec:infinite-state-graph}, we introduce a result that examines the structure of the state graph for an infinite number of balls and establishes a bijection between classes of prime patterns.

\section{Asymptotics for \texorpdfstring{$c_t(n)$}{c_t(n)}} \label{sec:asymptotics}

\subsection{Preliminaries}

We start with some definitions that can be found in \cite{Banaian2016}. We use these to develop a general theory of functions that define the behavior of juggling functions.

\begin{definition} \label{def:c_t}
For integers $t \ge 1$ and $n \ge 1$, the function $c_t(n)$ is defined as:
\[
c_t(n) = \sum_{\substack{p_1>\cdots>p_t \ge 1 \\ p_1+\cdots+p_t=n}}
\frac{1}{t(t+1)}\prod_{i=1}^t\bigg(\frac{i+1}{i}\bigg)^{p_i}
\]
\end{definition}

The asymptotic behavior of $c_t(n)$ is controlled by two sequences, $q_t$ and $r_t$.

\begin{definition} \label{def:q_t}
The sequence $q_t$ is defined by $q_1 = 1/2$ and for $t \ge 2$:
\[
q_t = \bigg(\frac{t-1}{2^t-t-1}\bigg) q_{t-1} = \frac{1}{2} \prod_{i=2}^t \frac{i-1}{2^i - i -1}
\]
\end{definition}

\begin{definition}
The sequence $r_t$ is defined by the recurrence $r_1 = 0$, $r_2 = \frac{4 \sqrt{3}}{9}$, and for $t \ge 3$:
\[
r_t = \frac{t-1}{\sqrt{3}^t - t - 1}r_{t-1} + 2 \left( \frac{2^t}{t+1} \right)^{(t-1)/2} q_{t-1}
\]
\end{definition}

These sequences provide sharp bounds for $c_t(n)$.

\begin{lemma}\label{lem:ct_bounds}
For all integers $t \ge 1$ and $n \ge 1$, the following inequality holds:
\[
q_t 2^n - r_t \sqrt{3}^n \leq c_t(n) \leq q_t 2^n
\]
\end{lemma}

Summing $c_t(n)$ over all possible values of $t$ gives the function $P'(2,n)$.

\begin{definition}
\[
P'(2, n) = \sum_{t=1}^\infty c_t(n)
\]
\end{definition}

The bounds on $c_t(n)$ lead to the asymptotic behavior of $P(2,n)$.

\begin{theorem}\label{thm:P2n_asymptotics}
Let $\gamma = \sum_{t=1}^\infty q_t$. The function $P'(2,n)$ has the asymptotic behavior:
\[
(\gamma - o_n(1)) 2^n \leq P'(2, n) \leq \gamma 2^n
\]
\end{theorem}

\subsection{General Asymptotic Theorem} 

\begin{theorem}[General Asymptotic Theorem]\label{thm:general_asymptotics}
Let $w: \mathbb{Z}^+ \to \mathbb{R}_{\ge 0}$ be a sequence of non-negative weights. Define the constant $\gamma_w = \sum_{t=1}^\infty w(t) q_t$, and assume this series converges. Let $F_w(n) = \sum_{t=1}^\infty w(t) c_t(n)$. Then, as $n \to \infty$:
\begin{enumerate}
    \item $F_w(n) = (\gamma_w - o(1)) 2^n$.
    \item $S_w(n) := \sum_{m=1}^{n-1} F_w(m) = (\gamma_w - o(1)) 2^n$.
    \item $(F_w * F_w)(n) := \sum_{m=1}^{n-1} F_w(m) F_w(n-m) = (\gamma_w^2 - o(1)) n 2^n$.
\end{enumerate}
\end{theorem}

\begin{proof}
The proof relies on the bounds for $c_t(n)$ from \cref{lem:ct_bounds}:
\[
q_t 2^n - r_t \sqrt{3}^n \leq c_t(n) \leq q_t 2^n.
\]

First, we work on the asymptotics of $F_w(n)$. We start by establishing the upper bound. Since $w(t) \ge 0$ for all $t$, we have:
\[
F_w(n) = \sum_{t=1}^\infty w(t) c_t(n) \leq \sum_{t=1}^\infty w(t) (q_t 2^n) = \left(\sum_{t=1}^\infty w(t) q_t\right) 2^n = \gamma_w 2^n.
\]
This implies $\limsup_{n \to \infty} \frac{F_w(n)}{2^n} \le \gamma_w$.

For the lower bound, let $\varepsilon > 0$. Since the series for $\gamma_w$ converges, we can choose an integer $K$ large enough such that $\sum_{t=1}^K w(t) q_t > \gamma_w - \frac{\varepsilon}{2}$. Let $\gamma_{w,K} = \sum_{t=1}^K w(t) q_t$ and $\rho_{w,K} = \sum_{t=1}^K w(t) r_t$. Note that $\rho_{w,K}$ is a finite sum and thus well-defined.

Since $w(t) \ge 0$ and $c_t(n) \ge 0$, we can truncate the sum for a lower bound:
\begin{align*}
F_w(n) = \sum_{t=1}^\infty w(t) c_t(n) &\ge \sum_{t=1}^K w(t) c_t(n) \\
&\ge \sum_{t=1}^K w(t) (q_t 2^n - r_t \sqrt{3}^n) \\
&= \left(\sum_{t=1}^K w(t) q_t\right) 2^n - \left(\sum_{t=1}^K w(t) r_t\right) \sqrt{3}^n \\
&= \gamma_{w,K} 2^n - \rho_{w,K} \sqrt{3}^n.
\end{align*}
Using our choice of $K$, we have $F_w(n) > (\gamma_w - \frac{\varepsilon}{2}) 2^n - \rho_{w,K} \sqrt{3}^n$.
We want to show that for sufficiently large $n$, this is greater than $(\gamma_w - \varepsilon) 2^n$. This is equivalent to showing $\frac{\varepsilon}{2} 2^n > \rho_{w,K} \sqrt{3}^n$, which simplifies to $\frac{\varepsilon}{2\rho_{w,K}} > (\frac{\sqrt{3}}{2})^n$. Since $\frac{\sqrt{3}}{2} < 1$, the right-hand side approaches $0$ as $n \to \infty$. Thus, for any fixed $\varepsilon$ and $K$, there exists an $N$ such that the inequality holds for all $n > N$.
This shows that for any $\varepsilon > 0$, $F_w(n) > (\gamma_w - \varepsilon) 2^n$ for sufficiently large $n$. This implies $\liminf_{n \to \infty} \frac{F_w(n)}{2^n} \ge \gamma_w$.

Combining the limsup and liminf results, we conclude that $\lim_{n \to \infty} \frac{F_w(n)}{2^n} = \gamma_w$.

Second, we work on the asymptotics of $S_w(n)$. For the upper bound, we sum the upper bound for $F_w(m)$:
\[
S_w(n) = \sum_{m=1}^{n-1} F_w(m) \le \sum_{m=1}^{n-1} \gamma_w 2^m = \gamma_w (2^n - 2) \le \gamma_w 2^n.
\]
This implies $\limsup_{n \to \infty} \frac{S_w(n)}{2^n} \le \gamma_w$.

For the lower bound, we use the same truncation argument. For a given $\varepsilon > 0$, choose $K$ such that $\gamma_{w,K} = \sum_{t=1}^K w(t) q_t > \gamma_w - \frac{\varepsilon}{2}$.
\begin{align*}
S_w(n) = \sum_{m=1}^{n-1} F_w(m) &\ge \sum_{m=1}^{n-1} \left( \sum_{t=1}^K w(t) c_t(m) \right) \\
&\ge \sum_{m=1}^{n-1} \left( \gamma_{w,K} 2^m - \rho_{w,K} \sqrt{3}^m \right) \\
&= \gamma_{w,K} \sum_{m=1}^{n-1} 2^m - \rho_{w,K} \sum_{m=1}^{n-1} \sqrt{3}^m \\
&= \gamma_{w,K} (2^n - 2) - \rho_{w,K} \frac{\sqrt{3}((\sqrt{3})^{n-1}-1)}{\sqrt{3}-1}.
\end{align*}
Dividing by $2^n$, we get
\[
\frac{S_w(n)}{2^n} \ge \gamma_{w,K} (1 - 2/2^n) - \frac{\rho_{w,K}}{\sqrt{3}-1} \left( \left(\frac{\sqrt{3}}{2}\right)^n - \frac{\sqrt{3}}{2^n} \right).
\]
Taking the limit inferior as $n \to \infty$:
\[
\liminf_{n \to \infty} \frac{S_w(n)}{2^n} \ge \gamma_{w,K} (1 - 0) - \frac{\rho_{w,K}}{\sqrt{3}-1} (0 - 0) = \gamma_{w,K}.
\]
Since we can choose $K$ such that $\gamma_{w,K}$ is arbitrarily close to $\gamma_w$, it follows that $\liminf_{n \to \infty} \frac{S_w(n)}{2^n} \ge \gamma_w$.

Combining the limsup and liminf results, we conclude that $\lim_{n \to \infty} \frac{S_w(n)}{2^n} = \gamma_w$.

Finally, we study the asymptotics of $(F_w * F_w)(n)$. The upper bound is found by substituting the upper bound for $F_w$:
\[
(F_w * F_w)(n) = \sum_{m=1}^{n-1} F_w(m) F_w(n-m) \le \sum_{m=1}^{n-1} (\gamma_w 2^m) (\gamma_w 2^{n-m}) = \sum_{m=1}^{n-1} \gamma_w^2 2^n = (n-1)\gamma_w^2 2^n.
\]
Dividing by $n 2^n$ gives $\frac{(F_w * F_w)(n)}{n 2^n} \le \frac{n-1}{n} \gamma_w^2$. Taking the limit superior, we get $\limsup_{n \to \infty} \frac{(F_w * F_w)(n)}{n 2^n} \le \gamma_w^2$.

For the lower bound, let $\varepsilon > 0$ and choose $K$ such that $\gamma_{w,K} > \gamma_w - \varepsilon$. Since $F_w(k) \ge 0$, we have $(F_w * F_w)(n) \ge (F_{w,K} * F_{w,K})(n)$, where $F_{w,K}(k) = \sum_{t=1}^K w(t) c_t(k)$.
Let $F_{w,K}(k) = \gamma_{w,K} 2^k - E_K(k)$, where $0 \le E_K(k) \le \rho_{w,K} \sqrt{3}^k$.
\begin{align*}
(F_{w,K} * F_{w,K})(n) &= \sum_{m=1}^{n-1} (\gamma_{w,K} 2^m - E_K(m))(\gamma_{w,K} 2^{n-m} - E_K(n-m)) \\
&= \sum_{m=1}^{n-1} \gamma_{w,K}^2 2^n - \gamma_{w,K} \sum_{m=1}^{n-1} (2^m E_K(n-m) + 2^{n-m} E_K(m)) + \sum_{m=1}^{n-1} E_K(m)E_K(n-m).
\end{align*}
The main term is $(n-1)\gamma_{w,K}^2 2^n$. The error terms are of a smaller order. The second term is bounded by $O(2^n)$ and the third by $O(n\sqrt{3}^n)$.
So, $(F_{w,K} * F_{w,K})(n) = (n-1)\gamma_{w,K}^2 2^n - O(2^n) - O(n\sqrt{3}^n)$.
Dividing by $n 2^n$:
\[
\frac{(F_{w,K} * F_{w,K})(n)}{n 2^n} = \frac{n-1}{n}\gamma_{w,K}^2 - O(1/n) - O((\sqrt{3}/2)^n).
\]
Taking the limit as $n \to \infty$, we get $\lim_{n \to \infty} \frac{(F_{w,K} * F_{w,K})(n)}{n 2^n} = \gamma_{w,K}^2$.
Therefore,
\[
\liminf_{n \to \infty} \frac{(F_w * F_w)(n)}{n 2^n} \ge \lim_{n \to \infty} \frac{(F_{w,K} * F_{w,K})(n)}{n 2^n} = \gamma_{w,K}^2 > (\gamma_w - \varepsilon)^2.
\]
Since this holds for any $\varepsilon > 0$, we can let $\varepsilon \to 0$ to find that $\liminf_{n \to \infty} \frac{(F_w * F_w)(n)}{n 2^n} \ge \gamma_w^2$.

Combining the limsup and liminf results, we conclude that $\lim_{n \to \infty} \frac{(F_w * F_w)(n)}{n 2^n} = \gamma_w^2$.
\end{proof}

\begin{corollary}\label{cor:general_asymptotics}
Let $w: \mathbb{Z}^+ \to \mathbb{R}_{\ge 0}$ be a sequence of non-negative weights. Define the constant $\gamma_w = \sum_{t=1}^\infty w(t) q_t$, and assume this series converges. Let $F_w(n) = \sum_{t=1}^\infty w(t) c_t(n)$. Then, as $n \to \infty$:
\begin{enumerate}
    \item $F_w(n) \sim \gamma_w 2^n$.
    \item $S_w(n) := \sum_{m=1}^{n-1} F_w(m) \sim \gamma_w 2^n$.
    \item $(F_w * F_w)(n) := \sum_{m=1}^{n-1} F_w(m) F_w(n-m) \sim \gamma_w^2 n 2^n$.
\end{enumerate}
\end{corollary}

\section{Multiplex 2-Ball} \label{sec:multiplex-2-ball}

\subsection{Preliminaries}

Multiplex juggling is a variation of normal juggling that allows for the simultaneous toss of multiple balls in a single time beat. Unlike standard juggling, where each throw and catch involves only one ball at a time, multiplex juggling introduces the possibility of throwing multiple balls within the same hand. In other words, we increase the capacity of our juggling hands, unlocking for example the transition $\langle 1, 0, 1 \rangle \to \langle 0, 2 \rangle$.

Concretely, we can set the $k$apacity of each hand to be $k \leq b$. Hence, a state $\sigma$ is not restricted anymore to having only 1's, as now, $a_i \in \{ 0, 1, 2, \dots, k \}$. However, we still have the condition that $\sum_i a_i = b$.

For example, if we set $b=3$ and $k=2$, we can have the state $\sigma = \langle 0, 2, 1 \rangle$. After one beat, all balls fall, so $\sigma \to \sigma'$ where $\sigma' = \langle 2, 1 \rangle$. Then, we have infinitely many choices for $\sigma''$ such that $\sigma' \to \sigma''$ since the two balls in the first position can go anywhere.  Some possible states include $\langle 1, 0, 2 \rangle$, $\langle 1, 1, 0, 1 \rangle$, and $\langle 2, 1 \rangle$. For all of them, $a_1 \geq 0$, and the only throw that would not be allowed is $\langle 3 \rangle$, as $k= 2$ does not allow $3$ balls to be in the same beat.

\cref{fig:b3-m2-diagram} shows some possible transitions in $\mathcal{M}(3,n,2)$.

\begin{figure}[htb]
\centering
\includegraphics[scale=1.25]{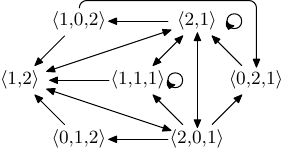}
\caption{A portion of the state diagram when $b=3$ and $k=2$.}
\label{fig:b3-m2-diagram}
\end{figure}

This section will focus in 2-ball multiplex prime patterns. We aim to expand the method used in \cite{Banaian2016} in order to find a count for $M'(2,n)$. 

As we are only going to look at $b=2$, we will assume that $k$ is always $2$ (as $k= 1$ is equivalent to normal juggling). Hence, all states will contain either one $2$ and possibly $0$'s or two $1$'s and possibly $0$'s. For example, this would be a valid 2-ball multiplex juggling pattern:
\begin{align*}
    &\langle 0, 0, 2 \rangle \to \langle 0, 2 \rangle \to \langle 2 \rangle \to \langle 0, 1, 0, 0, 1 \rangle \to \langle 1, 0, 0, 1 \rangle \\ 
    & \to  \langle 0, 1, 1 \rangle \to \langle 1, 1 \rangle \to \langle 1, 0, 0, 0, 1 \rangle \to \langle 0, 0, 0, 2 \rangle \to \langle 0, 0, 2 \rangle
\end{align*}

For the graphical representation, we will add double lines and double circles to represent states and transitions with two balls. As an example, \cref{fig:mpx-example} shows the pattern described above, but graphically.

\begin{figure}[H]
\centering
\begin{tikzpicture}
\multiplexexample
\end{tikzpicture}
\caption{An example of a 2-ball multiplex pattern.}
\label{fig:mpx-example}
\end{figure}
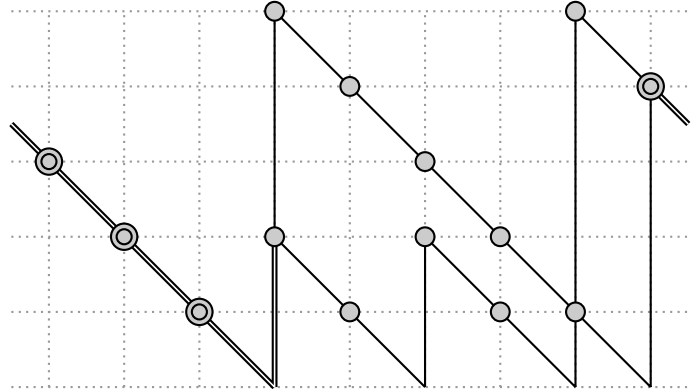

Furthermore, we will introduce a new notation for multiplex 2-ball states that will simplify future processes:
\[
    \sigma_{(i,j)} = \langle \underbrace{
    \underbrace{0, \dots, 0, 1}_{i}, 0, \dots, 0, 1}_{j}\rangle
\]
where $1 \leq i \leq j$. Essentially, each ``coordinate'' denotes the position of one ball. The definition shows the case when $i < j$ but in multiplex juggling, we can define that $i = j$ when both the first and second balls are indistinguishably in the same position:
\[
    \sigma_{(i,i)} = \langle \underbrace{
    \underbrace{0, \dots, 0, 2}_{i}}_{i}\rangle
\]

We propose a generalization of this notation that could be useful for future analysis of multiplex $b$-ball prime patterns:
\[
    \sigma_{(i, j, k, \dots)} = \langle \underbrace{\underbrace{\underbrace{\underbrace{0, \dots, 0, 1}_{i}, 0, \dots, 0, 1}_{j}, 0, \dots, 0, 1}_{k}, \dots}_{\dots}
    \rangle
\]

\subsection{Set Notation}

In this section, we will follow a similar process to the one used to solve the number of normal 2-ball juggling patterns. In order to simplify, we will often omit the fact that all patterns in this section are 2-ball patterns.

A first realization is that $\mathcal{N}'(2,n) \subseteq \mathcal{M}'(2,n)$, because any normal pattern can be considered a multiplex pattern and primeness is preserved.  This leads to the following definition:

\begin{definition} \label{def:mpx-strict}
    Let $\rho$ denote a juggling pattern. The \textbf{set of strict-multiplex patterns}, denoted $\mathcal{M}'(2,n)$, is defined as follows:
    \[
        \mathcal{M}' \setminus \mathcal{N}' (2,n) = \{ \rho \in \mathcal{M}'(2,n) \mid \rho \notin \mathcal{N}'(2,n) \},
    \]
    where $\mathcal{M}'(2,n)$ represents the set of all multiplex 2-ball prime patterns, and $\mathcal{N}'(2,n)$ represents the set of all normal 2-ball prime patterns.
\end{definition}

This allows us to count only the new patterns added to $\mathcal{M}'(2,n)$ by increasing hand capacity. Now, we introduce a fundamental property of these multiplex patterns.

\begin{lemma} \label{lem:mpx-2-state}
    A 2-ball strict multiplex prime pattern contains the state $\langle 2 \rangle$ exactly once.
\end{lemma}

\begin{proof}
    Let $\rho$ be a strict multiplex prime pattern. By primeness, each state $\langle 2 \rangle$ can appear at most once within $\rho$. Now it remains to prove that it appears at least once.

    Suppose initially that no state in $\rho$ takes the form $\sigma_{i,i}$. Then, every state must be of the form $\sigma_{i,j}$ where $i < j$. This condition implies that $\rho \in \mathcal{N}'(2,n)$ and consequently $\rho \notin \mathcal{M}' \setminus \mathcal{N}'(2,n)$, which leads to a contradiction.

    Therefore, there exists at least one state $\sigma_{i,i} \in \rho$ for some $i$. If $i = 1$, then $\sigma_{1,1} = \langle 2 \rangle$ appears in $\rho$. If not, the only possible transition from $\sigma_{i,i}$ is $\sigma_{i,i} \to \sigma_{i-1,i-1} \to \dots \to \sigma_{1,1} = \langle 2 \rangle$.
\end{proof}

\cref{lem:mpx-2-state} implies that a strict multiplex pattern can never be written using $C_0$, $C_1$ and $C_2$ cards, because we need cards that represent throws involving two balls. 

We introduce four new cards, $D_0$, $D_a$, $D_b$ and $D_c$ that represent all possible throws between states involving multiplex with two balls. These cards are graphically represented in \cref{fig:mpx-cards}, where a thick line represents two balls being in the same beat.

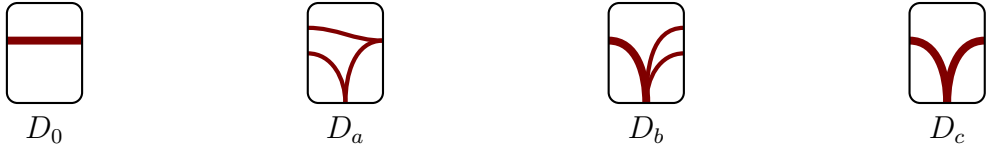
\begin{figure}[h]
\centering
\begin{tikzpicture}
\cardE{-4}0
\node at (-3.5,0) [below = 1] {$D_0$};
\cardD{0}0
\node at (0.5,0) [below = 1] {$D_a$};
\cardF{4}0
\node at (4.5,0) [below = 1] {$D_b$};
\cardG{8}0
\node at (8.5,0) [below = 1] {$D_c$};
\end{tikzpicture}
\caption{New cards used in strict multiplex patterns}
\label{fig:mpx-cards}
\end{figure}


\begin{itemize}
    \item $D_0$ represents the transition $\sigma_{(i,i)} \to \sigma_{(i-1,i-1)}$.
    \item $D_a$ represents the transition $\sigma_{(1,i)} \to \sigma_{(i-1,i-1)}$.
    \item $D_b$ represents the transition $\sigma_{(1,1)} \to \sigma_{(i,j)}$, with $i < j$.
    \item $D_c$ represents the transition $\sigma_{(1,1)} \to \sigma_{(i,i)}$.
\end{itemize}

The $D_c$ card is a rare card among prime patterns. Given a fixed length $n$, there is exactly one $\rho \in \mathcal{M}'\setminus \mathcal{N}' (2,n)$ such that $D_c \in \rho$. That pattern is
\[
    \sigma_{(n,n)} \to \sigma_{(n-1,n-1)} \to \dots \to \sigma_{(1, 1)} \to \sigma_{(n,n)}
\]
Hence, we can count all patterns such that $D_c \notin \rho$. Any prime pattern will have a sequence of cards with the form $D_a, D_0, D_0, \dots, D_0, D_b$. Suppose we place those balls at the end of the pattern.

\subsection{Main Theorem}

\begin{theorem} \label{thm:mpx-strict-formula}
    Let $M' \setminus N' (n,2)$ be the number of 2-ball strict multiplex prime juggling patterns of period $n$. Then
    \[
        M'(n,2) = \sum_{t} \left( \sum_{m=1}^{n-1} t \cdot c_t(m) \right) + 1
    \]
\end{theorem}

\begin{proof}[Proof of \cref{thm:mpx-strict-formula}]
    Let $\rho \in \mathcal{M}' \setminus \mathcal{N}' (n, 2)$. Notice that we can see this juggling pattern that begins with the state $\sigma_{(1, 1)}$, immediately after it transitions to $\sigma_{(i, j)}$ and eventually reaches $\sigma_{(1,j - i +1)}$. Then behaves as a normal prime pattern, and we can describe this behavior through sets $S_1, S_2, \dots, S_t$. Finally, when the smallest space of the last spacing happens, say a space $k$, the transition $\sigma_{(1,1+k)} \to \sigma_{(k,k)}$, which eventually reaches $\sigma_{(1,1)}$ and closes the cycle. 

    Notice how here, if $i = 1$, we started with a pattern $\rho \in \mathcal{M}' \setminus \mathcal{N}' (n, 2)$ and we obtained sets $S_1, S_2, \dots, S_t$ that define a pattern $\rho' \in \mathcal{N}'(n - 1, 2)$ (as we remove the multiplex throw). However, for any $i$, a pattern $\rho \in \mathcal{M} \setminus \mathcal{N}' (n, 2)$ gives  sets $S_1, S_2, \dots, S_t$ that define a pattern $\rho' \in \mathcal{N}'(n - i, 2)$.

    Reasoning backwards, for all $i > 0$ we can take a pattern $\rho' \in \mathcal{N}'(n - i, 2)$, obtain its sets $S_1, S_2, \dots, S_t$ and place the multiplex throw in $t$ different positions (because of the cyclic structure), obtaining $t$ different patterns $\rho_j \in \mathcal{M}' \setminus \mathcal{N}' (n, 2)$ for $j = 1, 2, \dots, t$. The theorem follows directly from the uniqueness of these $\rho_j$ given an initial $\rho'$, and the $+1$ comes from the pattern $\sigma_{(n,n)} \to \dots \to \sigma_{(1,1)}$.
\end{proof}

Given \cref{thm:mpx-strict-formula}, the following follows directly

\begin{corollary} \label{cor:mpx-formula}
    Let $M'(n,2)$ be the number of 2-ball multiplex prime juggling patterns of period $n$. Then
    \[
        M'(n,2) = \sum_t \left( c_t(n) + \sum_{m=1}^{n-1} t \cdot c_t(m) \right) + 1
    \]
\end{corollary}

\subsection{Asymptotics}

We present the exact values of $M'(n,2)$ for small $n$ in \cref{tab: mpx-data}.

\begin{table}[H]
    \centering
    \begin{tabular}{c@{\hskip 1in}c@{\hskip 1in}c}
        \begin{tabular}{|c|c|}
            \hline
            $n$ & $M'(n,2)$ \\
            \hline
            1 & 2 \\ \hline
            2 & 4 \\ \hline
            3 & 9 \\ \hline
            4 & 20 \\ \hline
            5 & 45 \\ \hline
            6 & 100 \\ \hline
            7 & 223 \\ \hline
            8 & 484 \\ \hline
            9 & 1053 \\ \hline
            10 & 2258 \\ \hline
        \end{tabular}
        &
        \begin{tabular}{|c|c|}
            \hline
            $n$ & $M'(n,2)$ \\
            \hline
            11 & 4827 \\ \hline
            12 & 10198 \\ \hline
            13 & 21505 \\ \hline
            14 & 44920 \\ \hline
            15 & 93687 \\ \hline
            16 & 194072 \\ \hline
            17 & 401061 \\ \hline
            18 & 824710 \\ \hline
            19 & 1693027 \\ \hline
            20 & 3460930 \\ \hline
        \end{tabular}
        &
        \begin{tabular}{|c|c|}
            \hline
            $n$ & $M'(n,2)$ \\
            \hline
            21 & 7064961 \\ \hline
            22 & 14377628 \\ \hline
            23 & 29219511 \\ \hline
            24 & 59240884 \\ \hline
            25 & 119980813 \\ \hline
            26 & 242531914 \\ \hline
            27 & 489839523 \\ \hline
            28 & 987879134 \\ \hline
            29 & 1990834305 \\ \hline
            30 & 4007533072 \\ \hline
        \end{tabular}
    \end{tabular}
    \caption{$M'(n,2)$ for $1 \leq n \leq 30$}
    \label{tab: mpx-data}
\end{table}

We now determine the asymptotic behavior of $M'(n,2)$ by applying the General Asymptotic Theorem (\cref{thm:general_asymptotics}) and its corollary (\cref{cor:general_asymptotics}).

Recall from \cref{thm:mpx-strict-formula} that the number of strict multiplex patterns is given by
\[
M' \setminus N' (n,2) = \sum_{m=1}^{n-1} \left( \sum_{t=1}^\infty t \cdot c_t(m) \right) + 1
\]
We define the weight function $w(t) = t$. The associated constant is
\[
\gamma_{M' \setminus N'} = \sum_{t=1}^\infty t \cdot q_t
\]
Using the definition of $q_t$, this series converges. We identify the inner sum of the expression for $M' \setminus N' (n,2)$ as $F_w(m)$ and the outer sum as $S_w(n)$ from \cref{cor:general_asymptotics}.

\begin{lemma} \label{lem:mpx-strict-asymptotics}
The number of strict multiplex prime patterns satisfies
\[
M' \setminus N' (n,2) \sim \gamma_{M' \setminus N'} 2^n
\]
as $n \to \infty$.
\end{lemma}

\begin{proof}
Let $w(t) = t$. By definition, $F_w(m) = \sum_{t=1}^\infty t c_t(m)$. The expression for strict multiplex patterns is
\[
M' \setminus N' (n,2) = \sum_{m=1}^{n-1} F_w(m) + 1 = S_w(n) + 1
\]
By \cref{cor:general_asymptotics}, part 2, we have $S_w(n) \sim \gamma_{M' \setminus N'} 2^n$. Since the constant term $1$ is negligible compared to $2^n$, the result follows.
\end{proof}

We now consider the total count of multiplex prime patterns, $M'(n,2)$.

\begin{theorem}\label{thm:mpx-asymptotics}
The number of multiplex 2-ball prime juggling patterns satisfies
\[
M'(n,2) \sim \gamma_{M'} 2^n
\]
where
\[
\gamma_{M'} = \sum_{t=1}^\infty (t+1) q_t
\]
\end{theorem}

\begin{proof}
By definition, $M'(n,2) = M' \setminus N'(n,2) + N'(n,2)$.
For the term $N'(n,2) = \sum_{t=1}^\infty c_t(n)$, we apply \cref{cor:general_asymptotics}, part 1, with the weight function $v(t) = 1$. This yields
\[
N'(n,2) \sim \left(\sum_{t=1}^\infty 1 \cdot q_t\right) 2^n = \gamma_{N'} 2^n
\]
Combining this with \cref{lem:mpx-strict-asymptotics}, we obtain
\[
M'(n,2) \sim \gamma_{M' \setminus N'} 2^n + \gamma_{N'} 2^n = (\gamma_{M' \setminus N'} + \gamma_{N'}) 2^n
\]
The constant is given by
\[
\gamma_{M'} = \sum_{t=1}^\infty t q_t + \sum_{t=1}^\infty q_t = \sum_{t=1}^\infty (t+1) q_t
\]
This completes the proof.
\end{proof}

We can compute the numerical value of the constant $\gamma_{M'}$ using the formula for $q_t$:
\[
\gamma_{M'} = \frac{1}{2}(1+1) + \sum_{t=2}^\infty (t+1) \frac{1}{2} \prod_{i=2}^t \frac{i-1}{2^i - i - 1}
\]
Using the identity $t! = t \prod_{i=2}^t (i-1)$, we can rewrite the terms to facilitate computation.

\section{Colored 2-Ball} \label{sec:colored-2-ball}

\subsection{Preliminaries}

Colored juggling is a variation of normal juggling where each ball is assigned a color. Unlike standard juggling, where the balls are typically indistinguishable from one another, colored juggling requires to not only keep track of the number of balls but also their specific colors. For example, in
\[
    \langle 0, 0, \textcolor{orange}{1}, 0, 0, \textcolor{blue}{1} \rangle \rightarrow \langle 0, \textcolor{orange}{1}, 0, 0, \textcolor{blue}{1}\rangle
\]
\[
    \langle 0, 0, \textcolor{orange}{1}, 0, 0, \textcolor{blue}{1}  \rangle \not\rightarrow \langle 0, \textcolor{blue}{1}, 0, 0, \textcolor{orange}{1} \rangle
\]
the first transition is valid because each balls comes down one position after a beat. The second one, however, is not valid because the balls that come down don't match the original balls.

For colored 2-ball, the first observation we make is that for something being a pattern, it's not sufficient that the states match at the end with the beginning, as we also need that the colors of the balls match.

\begin{lemma} \label{lem:clr-2-ball-even-cards}
A 2-ball colored juggling pattern contains an even number of $C_2$ cards.
\end{lemma}

\begin{proof}
By describing a pattern using card notation, we are describing the relative position of the balls in each time beat. Note that the cards $C_0$ and $C_1$ do not change the relative position of the balls, and the $C_2$ always do it. Hence, in order to keep the relative order of the colored balls from the end to cycling to the beginning, an even number of $C_2$ cards are needed.
\end{proof}

Hence, a colored 2-ball pattern can be defined by a collection of $2n$ sets $S_1, S_2, \dots, S_{2t}$ up to rotational symmetry ($S_1 S_2 S_3$ is the same as $S_2 S_3 S_1$). The condition for primeness is that a state can be repeated only if the corresponding colors are different. For example $\langle 0, \textcolor{blue}{1}, 0, 0, \textcolor{orange}{1} \rangle \neq \langle 0, \textcolor{orange}{1}, 0, 0, \textcolor{blue}{1} \rangle$. Translating this into sets, for a prime pattern we know $a \in S_i, S_j \implies i \not\equiv j (\mod 2)$. Here, as proven in \cite{Banaian2016} $n = \sum_i^{2t} \max(S_i)$.

\begin{figure}[H]
\centering
\begin{tikzpicture}
    \coloredexample
    \draw [decorate,decoration={brace,amplitude=10pt,mirror}] 
    (-1.5,-0.25) -- (0.5,-0.25) node [black,midway,yshift=-0.6cm] {$S_1$};
    \draw [decorate,decoration={brace,amplitude=10pt,mirror}] 
    (0.5,-0.25) -- (2.5,-0.25) node [black,midway,yshift=-0.6cm] {$R_1$};
    \draw [decorate,decoration={brace,amplitude=10pt,mirror}] 
    (2.5,-0.25) -- (6.5,-0.25) node [black,midway,yshift=-0.6cm] {$S_2$};
    \draw [decorate,decoration={brace,amplitude=10pt,mirror}] 
    (6.5,-0.25) -- (8.5,-0.25) node [black,midway,yshift=-0.6cm] {$R_2$};

\end{tikzpicture}
\caption{An example of a colored pattern.}
\label{fig:clr-example}
\end{figure}
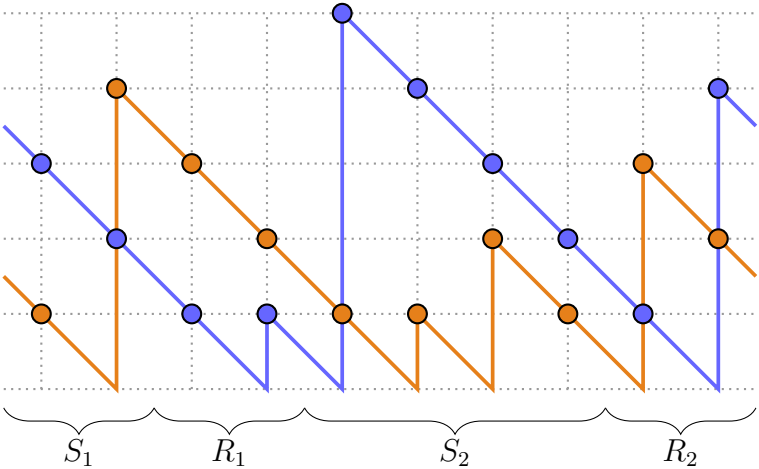

Another way to see this is renaming the initial sets: we can define two collections $\mathcal{S} = S_1, S_2, \dots, S_t$ and $\mathcal{R} = R_1, R_2, \dots, R_t$, which we assign to be the even and odd sets up to rotational symmetry. Here, the length of the pattern is
\[
    n = \sum_{i=1}^{t} \max(S_i) + \max(R_i)
\]
The condition for primeness for a colored prime pattern becomes
\[
    \forall \, i, j \quad S_i \cap S_j = \varnothing \quad \text{and} \quad R_i \cap R_j = \varnothing
\]

The above condition implies that, in a 2-ball prime colored pattern, both $\mathcal{S}$ and $\mathcal{R}$ have to represent normal prime juggling patterns with lengths $m$ and $n - m$, such that both have the same number of $C_2$ cards (which is equivalent as having the same number of partitions in the set notation), where $n$ is the length of the colored 2-ball prime juggling pattern.

Using this, we have converted out task of counting juggling patterns into a problem of counting collections of sets with certain conditions. We use slightly different methods for the cases when $n$ is even or $n$ is odd.

\subsection{Main Theorems}

\begin{theorem} \label{thm:clr-count-odd-length}
    The number of colored 2-ball prime juggling patterns of odd length is
    \[
        C'(2,n) = \sum_{m = 1}^{\frac{n-1}{2}} \sum_t t \cdot c_t(m) \cdot c_t(n-m)
    \]
\end{theorem}

\begin{proof}

We can assign each colored 2-ball prime juggling pattern to two collections $\mathcal{S}$ and $\mathcal{R}$ as defined above. Hence, we can count the number of ways to arrange those collections in order (up to cyclic permutation) to count the number of desired patterns.

First, we assume without loss of generality that the length of the pattern defined by $\mathcal{S}$ is $m$ and the length of the pattern defined by $\mathcal{R}$ is $n - m$ where $m < n-m$ (thus $m < \frac{n}{2} - 1$).

Given $m$, if we let $|\mathcal{S}| = |\mathcal{R}| = t$, we see there are $c_t(m)$ ways to arrange valid collections $\mathcal{S}$ and $c_t(n-m)$ ways to arrange valid collections $\mathcal{R}$.

As $\mathcal{S}$ and $\mathcal{R}$ describe different patterns, we know that there is at least one pair $i, j$ such that $S_i \neq R_j$. Furthermore, no two elements in $\mathcal{S}$ and $\mathcal{R}$ are equal. Hence, all $t$ rotations $\pi$ of $\mathcal{R}$ make the arrangement
\[
    S_1, R_{\pi(1)}, S_2, R_{\pi(2)}, \dots, S_t, R_{\pi(t)}
\]
different under rotational symmetry for each $\pi$, because the elements $S_i$ and $R_j$ will have different relative positions on it. Other way to see it is that we can place $R_1$ after $S_1$, or after $S_2$, etc, and each one will yield a different pattern.








    
    
    

\end{proof}

\begin{proposition} \label{prop:clr-count-even-length}
The number of colored 2-ball prime juggling patterns of even length is
\[
    C'(2,n) = \left( \sum_t \sum_{m = 1}^{\frac{n}{2}-1} \cdot c_t(m) \cdot t \cdot c_t(n-m) \right) + \left( \sum_t t \binom{c_t(\frac{n}{2})}{2} + \left\lceil \frac{t}{2} \right\rceil c_t\left(\frac{n}{2}\right) \right)
\]
\end{proposition}

\begin{proof}
The argument of the previous theorem is still valid for every $m < \frac{n}{2}$, which leads to the first term of the formula.

For the case $m = \frac{n}{2}$, for each partition of size $t$, there are $\binom{c_t\left(\frac{n}2\right)}{2}$ ways to create combinations of $(\mathcal{S}, \mathcal{R})$ with $\mathcal{S} \neq \mathcal{R}$. For each pair, there are $t$ ways to arrange each combination into a colored 2-ball prime juggling pattern.

This gives
\[
    \sum_t t \cdot \binom{c_t\left(\frac{n}{2}\right)}{2}
\]

Furthermore, there are $c_t\left(\frac{n}2\right)$ ways to create combinations of $(\mathcal{S}, \mathcal{R})$ with $\mathcal{S} = \mathcal{R}$, and in this case there are $\lceil \frac{t}2 \rceil$ ways to arrange each combination into a colored 2-ball prime juggling pattern. The reason for is, for $t$ even, $\frac{t}{2}$ out of the $t$ arrangements are equivalent to each other. For $t$ odd, this happens $\frac{t-1}{2}$ of the times, leaving $\frac{t+1}{2}$ valid placements, as shown in \cref{fig:clr-equal-arrangement}.


This gives the term
\[
    \sum_t \left\lceil \frac{n}{2} \right\rceil \cdot c_t\left(\frac{n}{2}\right) 
\]

\begin{center}
\begin{tikzpicture}

    \ucycle{0}{0}{5}{1.7}{0}
    \unode{0}{0}{5}{1.7}{0}
    \dnode{0}{0}{5}{1.7}{0}
    \dnode{0}{0}{5}{1.7}{1}
    \dnode{0}{0}{5}{1.7}{2}
    \ndnode{0}{0}{5}{1.7}{3}
    \ndnode{0}{0}{5}{1.7}{4}
    
    \draw[->, black!30!white] (-0.75,0.9) to[out=-40, in=-140] (0.75,0.9);
    \draw[->, black!30!white] (-1.1,-0.4) to[out=30, in=150] (1.1,-0.4);

    \def\x{7}

    \ucycle{\x}{0}{6}{1.7}{0}
    \unode{\x}{0}{6}{1.7}{0}
    \dnode{\x}{0}{6}{1.7}{0}
    \dnode{\x}{0}{6}{1.7}{1}
    \dnode{\x}{0}{6}{1.7}{2}
    \ndnode{\x}{0}{6}{1.7}{3}
    \ndnode{\x}{0}{6}{1.7}{4}
    \ndnode{\x}{0}{6}{1.7}{5}

    \draw[->, black!30!white] (\x-0.65,1) to[out=-50, in=-130] (\x+0.65,1);
    \draw[->, black!30!white] (\x-1.1,0) -- (\x+1.1,0);
    \draw[->, black!30!white] (\x-0.65,-1) to[out=50, in=130] (\x+0.65,-1);



\end{tikzpicture}
\label{fig:clr-equal-arrangement}
\end{center}

\end{proof}

\subsection{Asymptotics}

We now determine the asymptotic behavior of $C'(2,n)$. We first establish bounds for the case when $n$ is even to show that the parity of $n$ does not affect the leading asymptotic term.

\begin{lemma} \label{lem:clr-n-even-sandwich}
    For any $n$ even
    \[
        \frac{1}{2} \sum_{m=1}^{n-1} \sum_{t} t \cdot c_t(m) \cdot c_t(n-m) \leq C(2, n) \leq \sum_{m=1}^{\frac{n}{2}} \sum_{t} t \cdot c_t(m) \cdot c_t(n-m)
    \]
    holds.
\end{lemma}

\begin{proof}
    For $n$ odd, the sum ranges from $1$ to $(n-1)/2$, meaning there is no middle term $m = n/2$. The inequality becomes an equality, so the statement holds trivially.

    For $n$ even, the summation terms for $m \neq n/2$ are identical in all three expressions. Thus, it suffices to verify the inequality for the specific term where $m = n/2$. We must show:
    \[
        \frac{1}{2} t \left(c_t\left(\frac{n}{2}\right)\right)^2 \leq t \binom{c_t\left(\frac{n}{2}\right)}{2} + \left\lceil \frac{t}{2} \right\rceil c_t\left(\frac{n}{2}\right) \leq t \left(c_t\left(\frac{n}{2}\right)\right)^2
    \]
    Expanding the binomial coefficient in the middle term:
    \begin{align*}
        t \binom{c_t\left(\frac{n}{2}\right)}{2} + \left\lceil \frac{t}{2} \right\rceil c_t\left(\frac{n}{2}\right) 
        &= t \frac{c_t\left(\frac{n}{2}\right)\left(c_t\left(\frac{n}{2}\right) - 1\right)}{2} + \left\lceil \frac{t}{2} \right\rceil c_t\left(\frac{n}{2}\right) \\
        &= \frac{1}{2}t \left(c_t\left(\frac{n}{2}\right)\right)^2 - \frac{1}{2}t c_t\left(\frac{n}{2}\right) + \left\lceil \frac{t}{2} \right\rceil c_t\left(\frac{n}{2}\right) \\
        &= \frac{1}{2}t \left(c_t\left(\frac{n}{2}\right)\right)^2 + c_t\left(\frac{n}{2}\right) \left( \left\lceil \frac{t}{2} \right\rceil - \frac{t}{2} \right)
    \end{align*}
    For the lower bound, since $c_t(n/2) \ge 0$ and $\lceil t/2 \rceil \ge t/2$, the term $c_t(n/2) (\lceil t/2 \rceil - t/2)$ is non-negative. Thus:
    \[
        \frac{1}{2}t \left(c_t\left(\frac{n}{2}\right)\right)^2 \leq \frac{1}{2}t \left(c_t\left(\frac{n}{2}\right)\right)^2 + c_t\left(\frac{n}{2}\right) \left( \left\lceil \frac{t}{2} \right\rceil - \frac{t}{2} \right)
    \]
    For the upper bound, we can assume $c_t(n/2) \geq 1$ and $t \ge 1$ (because the inequality is trivially true when $c_t(n/2)=0$). We subtract the expanded middle term from the upper bound target $t (c_t(n/2))^2$:
    \begin{align*}
        &t \left(c_t\left(\frac{n}{2}\right)\right)^2 - \left( \frac{1}{2}t \left(c_t\left(\frac{n}{2}\right)\right)^2 + c_t\left(\frac{n}{2}\right) \left( \left\lceil \frac{t}{2} \right\rceil - \frac{t}{2} \right) \right) \\
        &= \frac{1}{2}t \left(c_t\left(\frac{n}{2}\right)\right)^2 - c_t\left(\frac{n}{2}\right) \left( \left\lceil \frac{t}{2} \right\rceil - \frac{t}{2} \right) \\
        &= c_t\left(\frac{n}{2}\right) \left( \frac{t}{2} c_t\left(\frac{n}{2}\right) - \left( \left\lceil \frac{t}{2} \right\rceil - \frac{t}{2} \right) \right)
    \end{align*}
    and we aim to show this is non-negative.
    
    We know that $\lceil t/2 \rceil - t/2$ is either $0$ (if $t$ is even) or $\frac12$ (if $t$ is odd). 
    Since $t \geq 1$ and $c_t(n/2) \geq 1$, we have $\frac{t}{2} c_t(n/2) \geq \frac12$. 
    Therefore, $\frac{t}{2} c_t(n/2) \ge \lceil t/2 \rceil - t/2$, making the difference non-negative. Thus:
    \[
        t \binom{c_t\left(\frac{n}{2}\right)}{2} + \left\lceil \frac{t}{2} \right\rceil c_t\left(\frac{n}{2}\right) \leq t \left(c_t\left(\frac{n}{2}\right)\right)^2
    \]
    Summing over all $t$ completes the proof.
\end{proof}

Furthermore, the difference between the odd and even formulas asymptotically relative to $n2^n$. Consequently, the asymptotic behavior of the function is independent of parity.

We now apply the General Asymptotic Theorem to find the explicit growth rate.

\begin{theorem}\label{thm:clr-asymptotics}
The number of colored 2-ball prime juggling patterns of length $n$ satisfies
\[
C'(2,n) \sim \gamma_{C'} n 2^n
\]
where
\[
\gamma_{C'} = \frac{1}{2} \sum_{t=1}^{\infty} t q_t^2 = \frac{1}{8} \sum_{t=1}^{\infty} t \left( \prod_{i=2}^{t} \frac{i-1}{2^i - i - 1} \right)^2 \approx 0.478326.
\]
\end{theorem}

\begin{proof}
    Let $S(n) = \sum_{t=1}^\infty t (c_t * c_t)(n)$.
    We apply \cref{thm:general_asymptotics} (Part 3) to the single function $c_t(n)$ (equivalent to setting the weight $w(k) = \delta_{kt}$). The theorem states that the convolution of $c_t$ with itself behaves as:
    \[
        (c_t * c_t)(n) \sim q_t^2 n 2^n.
    \]
    The term corresponding to the midpoint $m=n/2$ in the convolution is $c_t(n/2)^2$, which by \cref{lem:ct_bounds} is bounded by $O(2^n)$. This is negligible compared to the total convolution sum which grows as $O(n 2^n)$. Consequently, the difference between the upper and lower bounds in \cref{lem:clr-n-even-sandwich} vanishes asymptotically relative to the main term, and $C'(2,n) \sim \frac{1}{2} S(n)$ for both even and odd $n$.
    
    Summing the asymptotic contributions over $t$:
    \[
        S(n) = \sum_{t=1}^\infty t (c_t * c_t)(n) \sim \sum_{t=1}^\infty t (q_t^2 n 2^n) = \left( \sum_{t=1}^\infty t q_t^2 \right) n 2^n.
    \]
    The exchange of the summation and the limit is justified by the uniform bounds on $c_t(n)$ provided in \cref{lem:ct_bounds}, which ensure the series converges.
    Thus,
    \[
        C'(2,n) \sim \frac{1}{2} \left( \sum_{t=1}^\infty t q_t^2 \right) n 2^n = \gamma_{C'} n 2^n.
    \]
\end{proof}

\section{Passing 1-Ball and 2-Ball} \label{sec:passing}

A passing pattern is a juggling pattern in which balls can be thrown and caught with $k$ distinct hands, so that at most $k$ balls can be caught and thrown at one time, and they can be thrown to any hand.

In general, a state in passing is a $k \times \infty$ matrix $A$, where the sum of all entries is $b$, the number of balls. If an entry $a_{i,j} = 1$, it means that a ball is scheduled to land in $j$ beats at ``hand'' number $i$. Similarly to normal juggling, ones move to the left once each beat, and can be thrown anywhere when they reach the leftmost column (meaning that a ball has landed at a hand).

This is an example of a valid transition for $b = 2$, $k = 3$:
\[
\left\langle
\begin{matrix} 0 & 0 & 1 \\
0 & 1 & 0 \end{matrix} \right\rangle \rightarrow \left\langle
\begin{matrix} 0 & 1 & 0 \\
1 & 0 & 0 \end{matrix} \right\rangle
\]

Patterns are prime if they never repeat a state. In regular juggling, if a set of spacings between the balls occurs more than once, the pattern will not be prime as they will both eventually reach the state $\langle 1, \dots \rangle$.

For example, in 2-ball juggling, having a spacing of $j$ between the balls will eventually reach the state $\langle 1, \underbrace{0, \dots, 1}_j \rangle$ in both cases. However, in passing, this difference can occur if balls are scheduled to land at different hands. For example:
\[
\left\langle
\begin{matrix}
0 & 1 & 0 & 0 \\
0 & 0 & 0 & 1 \\
0 & 0 & 0 & 0
\end{matrix}
\right\rangle
\neq
\left\langle
\begin{matrix}
0 & 0 & 0 & 0 \\
0 & 1 & 0 & 0 \\
0 & 0 & 0 & 1
\end{matrix}
\right\rangle
\]
but both states have a ball landing in 2 beat and another ball landing in 4 beats.

In general, we denote by $P(b,n,k)$ the number of $b$-ball passing patterns with $k$ hands and length $n$, and we use $P(b,n,k)$ to denote the number of such prime patterns. Furthermore, we use $\mathcal{P}$ and $\mathcal{P'}$ to refer to the set of all such patterns, not the count of them: this is, we can talk about a pattern $\rho \in \mathcal{P}$.

\subsection{Passing 1-Ball} \label{sec:passing-1-ball}

In this section we find an exact count for $P'(1, n, k)$.

\begin{theorem} \label{thm:passing-1-ball}
The number of 1-ball prime passing patterns of length $n$ and $k$ hands is
\[
P'(1, n, k) = \sum_{h=1}^{k} \binom{k}{h} \binom{n-1}{h-1} (h - 1)!
\]
\end{theorem}

\begin{proof}

For passing patterns with a single ball, $k$ possible hands, and $n$ states, we can choose $h$ of the $k$ hands to throw to within the pattern. Notice that we may not throw to any given hand more than once in a one-ball pattern, because that would repeat the state of the ball landing in that hand.

Therefore, for each of $\binom{k}{h}$ choices, for each hand in $h$, the ball may be thrown to one height above that hand. As the heights of the throws must sum to the cycle length $n$ when we have only one ball, we get the heights by a partition of $n$ into $h$ parts. There are therefore (by applying Stars and Bars) $\binom{k}{h} \binom{n-1}{h-1}$ possible heights for any $1 \leq h \leq k$. 

Given that the pattern depends on the order of the heights up to rotation, of which there are $(h-1)!$ possibilities. Summing over all possible numbers of active hands $h$ from $1$ to $k$ gives the total count.

\end{proof}

\subsection{Asymptotic Behavior of \texorpdfstring{$P'(1, n, k)$}{P'(1, n, k)}}

In this section, we analyze the behavior of the function $P'(1, n, k)$. the formula derived in \cref{sec:passing-1-ball} is a polynomial in $k$, and for a fixed $k$, it is a polynomial in $n$.

\subsubsection{Case 1: Fixed \texorpdfstring{$n$}{n}, \texorpdfstring{$k \to \infty$}{k -> infinity}}

When $n$ is fixed, the number of hands $k$ can grow indefinitely. The sum is limited by $h \le n$, so:
\[
P'(1, n, k) = \sum_{h=1}^{n} \binom{k}{h} \binom{n-1}{h-1} (h - 1)!
\]

\begin{proposition}
For a fixed integer $n \ge 1$, as $k \to \infty$, the number of prime 1-ball, $k$-hand passing patterns of length $n$ behaves as:
\[
P'(1, n, k) \sim \frac{k^n}{n}
\]
\end{proposition}

\begin{proof}
Each term in the finite sum is a polynomial in $k$. The term $\binom{k}{h} = \frac{k(k-1)\cdots(k-h+1)}{h!}$ is a polynomial in $k$ of degree $h$. Therefore, the entire sum $P'(1, n, k)$ is a polynomial in $k$. The degree of this polynomial is determined by the largest value of $h$ in the sum, which is $h=n$.

We can analyze the term for $h=n$:
\begin{align*}
    \text{Term}_{h=n} &= \binom{k}{n} \binom{n-1}{n-1} (n - 1)! \\
    &= \binom{k}{n} \cdot 1 \cdot (n-1)! \\
    &= \frac{k(k-1)\cdots(k-n+1)}{n!} (n-1)! \\
    &= \frac{k(k-1)\cdots(k-n+1)}{n}
\end{align*}
This is a polynomial in $k$ of degree $n$. The leading term is $\frac{k^n}{n}$.

Now consider any other term for $h < n$. The degree of this term as a polynomial in $k$ is $h$. Since all other terms have a degree strictly less than $n$, the asymptotic behavior of $P'(1, n, k)$ for large $k$ is dominated by the $h=n$ term.

Now we can write the sum as:
\[
P'(1, n, k) = \frac{k(k-1)\cdots(k-n+1)}{n} + \sum_{h=1}^{n-1} \binom{k}{h} \binom{n-1}{h-1} (h - 1)!
\]
The first term is a polynomial in $k$ of degree $n$ with leading term $\frac{k^n}{n}$. The summation is a polynomial in $k$ of degree at most $n-1$. Therefore,
\begin{align*}
    P'(1, n, k) &= \left(\frac{1}{n}k^n + O(k^{n-1})\right) + O(k^{n-1}) \\
    &= \frac{1}{n}k^n + O(k^{n-1})
\end{align*}
as desired.
\end{proof}

\subsubsection{Case 2: Fixed \texorpdfstring{$k$}{k}, \texorpdfstring{$n \to \infty$}{n -> infinity}}

When $k$ is fixed, the number of available hands is constant, but the length of the pattern $n$ can grow. The sum is over a fixed number of terms, from $h=1$ to $h=k$, so we have:
\[
P'(1, n, k) = \sum_{h=1}^{k} \binom{k}{h} (h - 1)! \binom{n-1}{h-1}
\]

\begin{proposition}
For a fixed integer $k \ge 1$, as $n \to \infty$, the number of prime 1-ball, $k$-hand passing patterns of length $n$ behaves as:
\[
P'(1, n, k) \sim n^{k-1}
\]
\end{proposition}

\begin{proof}
Each term in the sum is a polynomial in $n$. The term $\binom{n-1}{h-1} = \frac{(n-1)(n-2)\cdots(n-h+1)}{(h-1)!}$ is a polynomial in $n$ of degree $h-1$. The degree of the entire sum is determined by the largest value of $h-1$, which occurs at $h=k$. Thus, $P'(1, n, k)$ is a polynomial in $n$ of degree $k-1$.

Let's analyze the term for $h=k$:
\begin{align*}
    \text{Term}_{h=k} &= \binom{k}{k} (k-1)! \binom{n-1}{k-1} \\
    &= 1 \cdot (k-1)! \cdot \frac{(n-1)(n-2)\cdots(n-k+1)}{(k-1)!} \\
    &= (n-1)(n-2)\cdots(n-k+1)
\end{align*}
This is a polynomial in $n$ of degree $k-1$. The leading term is $n^{k-1}$.

Now consider any other term for $h < k$. The degree of this term as a polynomial in $n$ is $h-1$. Since all other terms have a degree strictly less than $k-1$, the asymptotic behavior of $P'(1, n, k)$ for large $n$ is dominated by the $h=k$ term.

Now we can write the sum as:
\[
P'(1, n, k) = (n-1)(n-2)\cdots(n-k+1) + \sum_{h=1}^{k-1} \binom{k}{h} (h - 1)! \binom{n-1}{h-1}
\]
The first term is a polynomial in $n$ of degree $k-1$ with a leading coefficient of 1.
\begin{align*}
(n-1)(n-2)\cdots(n-k+1) &= n^{k-1} - (1+2+\dots+(k-1))n^{k-2} + O(n^{k-3}) \\
&= n^{k-1} - \frac{k(k-1)}{2}n^{k-2} + O(n^{k-3})
\end{align*}
The summation is a polynomial in $n$ of degree at most $k-2$, since the highest degree term in the sum comes from $h=k-1$, which is of degree $(k-1)-1 = k-2$.
Therefore,
\begin{align*}
    P'(1, n, k) &= \left(n^{k-1} + O(n^{k-2})\right) + O(n^{k-2}) \\
    &= n^{k-1} + O(n^{k-2})
\end{align*}
as desired.
\end{proof}

\subsection{Passing 2-Ball} \label{sec:passing-2-ball}

In this section, we find a lower bound for $P'(2, n, k)$. Notice that we can embed any prime multiplex pattern $\rho \in \mathcal{M}'(2, n, k)$ as a passing pattern. The choices of the hands doesn't affect primality, so we obtain the following:

\begin{theorem} \label{thm:passing-2-lower-bound}
    The number of passing 2-ball prime patterns using $k$ hands can be bounded by
    \[
    P'(2, n, k) \geq \sum_{\rho \in \mathcal{N}(2,n)} k^{\varphi(\rho)} + \sum_{\rho \in \mathcal{M}'(2,n)} k^{\varphi(\rho)} \cdot (k-1)    
    \]
    where $\varphi(\rho)$ gives the number of throw cards ($C_1$, $C_2$, $D_a$ or $D_b$) in a pattern.
\end{theorem}

\begin{proof}
For each $\rho \in \mathcal{N}(2,n)$, we can construct different passing patterns by choosing a hand for each throw that we make. Regardless of these choices, the patterns will remain prime as all states will be different. Furthermore, as there are $\varphi(\rho)$ throw cards, we can throw any ball to any of the $k$ hands in any throw. This argument gives the first sum of the lower bound.

For each $\rho \in \mathcal{M}(2,n)$, we can do the same thing for $C_1$ and $C_2$ cards. Then, at the moment of a $D_a$ throw, only $k-1$ hands can be chosen, as one is taken by the higher ball at that beat. From there, the state $\sigma_{(i,i})$ will be reached, followed by transitions up to the state $\langle 2 \rangle$ where a $D_b$ throw will be made. As both of the balls will go to different heights, there are $k^2$ possibilities for this throw. This gives the second term of the bound.
\end{proof}

But what is the behavior of the function $\varphi(\rho)$? At least, we know that $\varphi(\rho) \geq t$ where $t$ is the partition size of $\rho$. Therefore,
\[
    P'(2, n, k) \geq 
    \sum_{m < n} \sum_t (t c_t(m) + 1) \cdot k^t \cdot (k-1)
    + \sum_t c_t(n) \cdot k^{t}
\]

\subsection{Asymptotic Behavior of the Lower Bound for \texorpdfstring{$P'(1, n, k)$}{P'(2, n, k)}}

In this section, we analyze the asymptotic behavior of our lower bound  $P'(2, n, k)$, the number of 2-ball, $k$-hand prime passing patterns of length $n$, which begins with the lower bound established in \cref{thm:passing-2-lower-bound}.

We will analyze the behavior of this lower bound with fixed period $n$ as the number of hands $k$ grows, and then with a fixed number of hands $k$ as the period $n$ grows. 

\subsubsection{Case 1: Fixed \texorpdfstring{$n$}{n}, \texorpdfstring{$k \to \infty$}{k -> infinity}}

When the period $n$ is fixed and the number of hands $k$ tends to infinity, the behavior of $P'(2, n, k)$ is polynomial in $k$.

\begin{proposition}
For a fixed integer $n \ge 1$, as $k \to \infty$, the number of prime 2-ball, $k$-hand passing patterns of length $n$ is bounded below by:
\[
P'(2, n, k) \gtrsim k^{n+1}
\]
\end{proposition}

\begin{proof}

Since $n$ is fixed, the sets of prime patterns $\mathcal{N}(2,n)$ and $\mathcal{M}'(2,n)$ are finite. The expression on the right-hand side is a finite sum of polynomials in $k$. Specifically, for each pattern $\rho$, the term is proportional to $k^{\varphi(\rho)}$ or $k^{\varphi(\rho)+1}$.
Thus, $P'(2, n, k)$ is bounded below by a polynomial in $k$.
Let $d_n = \max \{ \varphi(\rho) + \delta_{\rho} \mid \rho \in \mathcal{N}'(2,n) \cup \mathcal{M}'(2,n) \}$, where $\delta_{\rho} = 1$ if $\rho$ is a strict multiplex pattern and $0$ otherwise.

Now consider the pattern $\rho^*$:
\[
\langle 2 \rangle \to \langle 1, \underbrace{0, \dots, 0}_{n-2}, 1\rangle \to \langle 1, \underbrace{0, \dots, 0}_{n-3}, 1\rangle  \to \dots \langle 1,1 \rangle \to \langle 2 \rangle
\]
This pattern is in $\mathcal{M}'(2, n)$ and all of its cards are throw cards. Hence $d_n = n$ and the term $k^d \cdot (k-1)$ will appear as part of the right sum in \cref{thm:passing-2-lower-bound}

\end{proof}

\subsubsection{Case 2: Fixed \texorpdfstring{$k$}{k}, \texorpdfstring{$n \to \infty$}{n -> infinity}}

For a fixed number of hands $k$ as the period $n$ grows, we utilize the General Asymptotic Theorem (\cref{thm:general_asymptotics}) to evaluate the sums appearing in the bound.

\begin{proposition}
For a fixed integer $k \ge 1$, as $n \to \infty$, the number of prime 2-ball, $k$-hand passing patterns of length $n$ is bounded below by:
\[
P'(2, n, k) \gtrsim \gamma_{P'}(k) 2^n
\]
where
\[
\gamma_{P'}(k) = \sum_{t=1}^\infty q_t k^t \big(1 + kt - t\big)
\]
where $q_t$ is defined as in \cref{def:q_t}
\end{proposition}

\begin{proof}
Using the inequality $\varphi(\rho) \geq t$, where $t$ is the number of sets in the partition associated with $\rho$, we can weaken the lower bound to a form suitable for our asymptotic tools:
\begin{align*}
    P'(2, n, k) &\geq \sum_{\rho \in \mathcal{N}(2,n)} k^{t} + (k-1) \sum_{\rho \in \mathcal{M}'(2,n)} k^{t} \\
    &= \sum_{t=1}^\infty k^t c_t(n) + (k-1) \left( \sum_{m=1}^{n-1} \sum_{t=1}^\infty t k^t c_t(m) + 1 \right)
\end{align*}
We define two weight functions: $w_1(t) = k^t$ and $w_2(t) = t k^t$.
The associated series of constants are:
\[
\gamma_{w_1} = \sum_{t=1}^\infty k^t q_t \quad \text{and} \quad \gamma_{w_2} = \sum_{t=1}^\infty t k^t q_t
\]
Since $q_t$ decays superexponentially (roughly as $2^{-t^2/2}$), both series converge for any fixed $k$.
Applying \cref{cor:general_asymptotics}:
\begin{itemize}
    \item The first term is $F_{w_1}(n) \sim \gamma_{w_1} 2^n$.
    \item The second term involves $S_{w_2}(n) = \sum_{m=1}^{n-1} F_{w_2}(m)$. By the theorem, $S_{w_2}(n) \sim \gamma_{w_2} 2^n$.
\end{itemize}
Combining these results, we obtain the asymptotic behavior of the lower bound:
\begin{align*}
    P'(2, n, k) &\gtrsim \gamma_{w_1} 2^n + (k-1) \gamma_{w_2} 2^n \\
    &= \left( \sum_{t=1}^\infty k^t q_t + (k-1) \sum_{t=1}^\infty t k^t q_t \right) 2^n \\
    &= \left( \sum_{t=1}^\infty q_t k^t \big(1 + kt -t\big) \right) 2^n
\end{align*}

\end{proof}

However, using $\varphi(\rho) \geq t$ greatly worsens the lower bound. The authors believe that some Theorem fundamentally similar to \cref{thm:general_asymptotics} should follow replacing $2^n$ with $b^n$. Hence, intuition tells us the following conjecture, which is stronger than what we were able to prove.

\begin{conjecture}

For a fixed integer $k \ge 1$, as $n \to \infty$, the number of prime 2-ball, $k$-hand passing patterns of length $n$ is bounded below by:
\[
P'(2, n, k) \gtrsim \gamma_{P'}(k) (1+k)^n
\]
for some constant $\gamma_{P'}(k)$.

\end{conjecture}

\section{Lower Bound on \texorpdfstring{$N'(b,n)$}{N'(b,n)}} \label{sec:better-b-ball-bounds}

In this section, we establish a new lower bound for $N'(b,n)$, the number of prime juggling patterns of length $n$ with $b$ balls. We generalize the lower bound found by \cite{Banaian2016}, which states that $N'(b, n) \geq \frac{1}{b} \cdot b^n$. We follow a similar approach, by finding a construction that generates a family of prime patterns. In order to count these, we associate them with ``filled'' Ferrers diagrams of partitions of $n$ into distinct parts.

\subsection{The Filled Ferrers Diagram Construction}

We begin by defining a structure that encodes the essential characteristics of our prime juggling patterns.

\begin{definition}
Let $\lambda = (p_1, p_2, \ldots, p_t)$ be a partition of $n$ into $t$ distinct parts, such that $p_1 > p_2 > \cdots > p_t \ge 1$. A \textbf{filled Ferrers diagram} associated with $\lambda$ is a filling of the cells of the Ferrers diagram of $\lambda$ with integers from the set $\{0, 1, \ldots, b\}$, subject to the following constraints:
\begin{enumerate}
    \item The rightmost cell of each row must contain the integer $b$.
    \item No other cell in the diagram may contain the integer $b$.
    \item Each column of the diagram must contain at most one integer from the set $\{b, b-1\}$.
    \item All other cells must be filled with an integer from $\{0, 1, \ldots, b-2\}$.
\end{enumerate}
\end{definition}

From a valid filled Ferrers diagram, we construct a sequence $w$ of length $n$, which we call the \textit{landing word}.

\begin{definition}[Landing Word Construction]
Given a filled Ferrers diagram for a partition $\lambda$ with $t$ rows:
\begin{enumerate}
    \item Choose a cyclic ordering of the $t$ rows.
    \item For each row, read its entries from right to left (reverse order).
    \item Concatenate these reversed rows according to the chosen cyclic ordering to form a word $w = w_1 w_2 \dots w_n \in \{0, \dots, b\}^n$.
\end{enumerate}
\end{definition}

This word $w$ determines the timing of ball landings in the pattern. Specifically, a non-zero entry $v$ at index $j$ implies that a ball lands at beat $j$ having been the $v$-th highest ball in the air at the moment it was thrown.

\subsection{From Words to Patterns}

We now define an algorithm to map a landing word $w$ to a juggling pattern $u = u_1 u_2 \dots u_n$, where each $u_j$ is a juggling card $C_k$. Recall that $C_0$ represents a beat where no ball lands, and $C_k$ for $k > 0$ represents a beat where a ball lands and is thrown to become the $k$-th ball in the air (relative order).

\begin{definition}[Pattern Generation Algorithm]
Let $w \in \{0, \dots, b\}^n$. We construct the pattern $u$ as follows:
\begin{enumerate}
    \item For every index $j$ where $w_j = 0$, assign $u_j = C_0$.
    \item For each value $v \in \{1, \dots, b\}$:
    \begin{enumerate}
        \item Identify all indices $J_v = \{j \mid w_j = v\}$.
        \item For each $j \in J_v$, we determine the beat $k$ where the ball landing at $j$ was thrown. To find $k$, start at index $j$ and move backwards (leftwards) through the word $w$ cyclically. Maintain a ``bump count'', initialized to 0.
        \item At each step moving left, examine the entry. If the entry corresponds to a beat that has not yet been assigned a throw card, or if it has been assigned a card $C_x$ where $x$ is greater than the current bump count, increment the bump count by 1.
        \item Continue moving left until the bump count reaches $v$. The current beat is the throw time. Assign the card $C_v$ to this beat.
    \end{enumerate}
\end{enumerate}
\end{definition}

This algorithm reconstructs the throw sequence required to produce the landing schedule specified by $w$. The ``bump count'' logic accounts for the fact that a ball thrown as the $v$-th highest must wait for $v-1$ other balls (those lower than it) to land and be re-thrown before it becomes the lowest ball and lands itself.

\begin{theorem} \label{thm:construction-validity}
    For any word $w$ generated from a valid filled Ferrers diagram, the Pattern Generation Algorithm produces a valid, prime $b$-ball juggling pattern of period $n$.
\end{theorem}

\begin{proof}

It is well known that any sequence of cards is a valid pattern.
    
For primality, the ``top gap'' of a state is the distance between the highest and second-highest balls. We aim to show that every pattern generated using this algorithm differs in the top gap at least once.

The top gap is fully determined by the placement of $C_b$ and $C_{b-1}$ throws. In our construction, the $C_b$ throws correspond to the $b$'s in the diagram (one per row), and the $C_{b-1}$ throws correspond to the $(b-1)$'s.
    
Constraint 3 ensures that in the diagram, the vertical alignment of $b$ and $b-1$ entries is unique for each column. 

\textcolor{red!70!black}{[A formal proof of this result is currently in preparation and will be included in a subsequent version of this preprint.]}

\end{proof}

\subsection{Enumeration}

We now count the number of distinct prime patterns generated by this construction.

\begin{lemma} \label{lem:diagram-count}
    Let $\lambda = (p_1, \dots, p_t)$ be a partition of $n$ into distinct parts. The number of valid filled Ferrers diagrams for $\lambda$ is
    \[
    D_\lambda = \frac{(b-1)^n}{\binom{t+b-1}{b} t!} \prod_{i=1}^t \left(\frac{i+b-1}{i+b-2}\right)^{p_i}
    \]
\end{lemma}

\begin{proof}

We count the valid fillings column by column. Let $h_j$ be the height of column $j$ in the Ferrers diagram. Since the parts are distinct, the column heights decrease in steps of 1. Specifically, for the columns $j$ in the range $[p_{i+1} + 1, p_i]$ (where $p_{t+1}=0$), the height is $h_j = i$.

Consider a column of height $h$. There are two possible cases:
\begin{itemize}
    \item \textbf{Case 1: A row ends in this column.} This occurs exactly when $j = p_i$ for some $i$. By Constraint 1, the cell $(i, p_i)$ must be filled with $b$. By Constraint 3, no other cell in this column can be $b$ or $b-1$. The remaining $h-1$ cells must be filled with values from $\{0, \dots, b-2\}$. There are $(b-1)^{h-1}$ ways to fill such a column.
    \item \textbf{Case 2: No row ends in this column.} This occurs for $j \in [p_{i+1}+1, p_i - 1]$. By Constraint 1, no cell is $b$. By Constraint 3, at most one cell is $b-1$.
    \begin{itemize}
        \item \textbf{Subcase 2a:} No cell is $b-1$. All $h$ cells are from $\{0, \dots, b-2\}$. There are $(b-1)^h$ ways to fill this.
        \item \textbf{Subcase 2b:} Exactly one cell is $b-1$. There are $h$ choices for the position, and the rest are from $\{0, \dots, b-2\}$. There are $h(b-1)^{h-1}$ ways to fill this.
    \end{itemize}
    Thus, there are $(b-1)^h + h(b-1)^{h-1} = (b-1)^{h-1}(b-1+h)$ ways to fill such a column.
\end{itemize}

For the range of columns corresponding to row $i$ (length $p_i - p_{i+1}$), we have 1 column of Case 1 and $p_i - p_{i+1} - 1$ columns of Case 2. The height is $h=i$. The number of ways for this range is:
\[
(b-1)^{i-1} \cdot \left[ (b-1)^{i-1}(b+i-1) \right]^{p_i - p_{i+1} - 1} = (b-1)^{(i-1)(p_i - p_{i+1})} (b+i-1)^{p_i - p_{i+1} - 1}
\]
Taking the product over all $i=1, \dots, t$:
\[
\prod_{i=1}^t (b-1)^{(i-1)(p_i - p_{i+1})} (b+i-1)^{p_i - p_{i+1} - 1}
\]
The exponent of $(b-1)$ sums to $\sum (i-1)(p_i - p_{i+1}) = \sum p_i - p_1 = n - p_1$.

The product of the $(b+i-1)$ terms can be rearranged using the fact that $p_{t+1}=0$:
\[
\prod_{i=1}^t (b+i-1)^{p_i - p_{i+1} - 1} = \frac{\prod_{i=1}^t (b+i-1)^{p_i}}{\prod_{i=1}^t (b+i-1)^{p_{i+1}} \prod_{i=1}^t (b+i-1)}
\]
The denominator term $\prod_{i=1}^t (b+i-1) = \frac{(t+b-1)!}{(b-1)!} = t! \binom{t+b-1}{b}$.
Shifting indices in the other denominator product, we get:
\[
\frac{1}{t! \binom{t+b-1}{b}} \cdot \frac{b^{p_1} \prod_{i=2}^t (b+i-1)^{p_i}}{\prod_{i=2}^t (b+i-2)^{p_i}} = \frac{1}{t! \binom{t+b-1}{b}} b^{p_1} \prod_{i=2}^t \left(\frac{b+i-1}{b+i-2}\right)^{p_i}
\]
Combining with the $(b-1)^{n-p_1}$ term:
\[
D_\lambda = \frac{(b-1)^n}{t! \binom{t+b-1}{b}} \left(\frac{b}{b-1}\right)^{p_1} \prod_{i=2}^t \left(\frac{b+i-1}{b+i-2}\right)^{p_i} = \frac{(b-1)^n}{t! \binom{t+b-1}{b}} \prod_{i=1}^t \left(\frac{i+b-1}{i+b-2}\right)^{p_i}
\]
\end{proof}

To obtain the total number of patterns, we sum over all partitions $\lambda$. For each partition, there are $t!$ distinct linear orderings of the rows and each ordering produces a word $w$. 

\begin{theorem} \label{thm:b-ball-lower-bound}
    The number of prime $b$-ball juggling patterns of period $n$ is bounded below by:
    \[
    N'(b,n) \ge \frac{1}{b} \sum_{t \ge 1} \sum_{\substack{p_1 > \cdots > p_t \ge 1 \\ p_1 + \cdots + p_t = n}} \frac{(b-1)^n}{\binom{t+b-1}{b}} \prod_{i=1}^{t} \left(\frac{i+b-1}{i+b-2}\right)^{p_i}
    \]
\end{theorem}

\begin{proof}

The term inside the sum corresponds to $t! \cdot D_\lambda$, which is the number of valid landing words $w$ that can be formed by concatenating the rows of the filled Ferrers diagrams for a fixed partition $\lambda$ in any linear order. Summing over all partitions gives the total number of such words.

\textcolor{red!70!black}{[A formal proof of this result is currently in preparation and will be included in a subsequent version of this preprint.]}

\end{proof}

\subsection{Asymptotic Analysis}

We now focus on the asymptotic behavior of the lower bound found in the previous section.

\begin{theorem} \label{thm:asymptotic}
    Let $b \ge 2$. Then,
    \[
    \frac1b\sum_{t\ge 1}
    \sum_{\substack{p_1>\cdots>p_t\ge 1\\p_1+\cdots+p_t=n}}\frac{(b-1)^n}{\binom{t+b-1}{b}}\prod_{i=1}^t\bigg(\frac{i+b-1}{i+b-2}\bigg)^{p_i}=\big(\gamma_b-o(1)\big)b^n
    \]
    where
    \[
    \gamma_b=\frac1b\sum_{t\ge1}\frac{(t-1)!(b-1)^{(t+2)(t-1)/2}}{\prod_{i=2}^t\big((b-1)b^i-(b-1)^i(b+i-1)\big)}
    \]
\end{theorem}

\begin{proof}
    Let $S_n(b)$ denote the sum on the left-hand side of the equation. We aim to prove that $\lim_{n \to \infty} \frac{S_n(b)}{b^n} = \gamma_b$.
    Let $A_n = \frac{S_n(b)}{b^n}$. Substituting the expression for $S_n(b)$, we have:
    \[
    A_n = \frac{1}{b^{n+1}} \sum_{t\ge 1} \frac{(b-1)^n}{\binom{t+b-1}{b}} \sum_{\substack{p_1>\cdots>p_t\ge 1\\p_1+\cdots+p_t=n}} \prod_{i=1}^t \left(\frac{i+b-1}{i+b-2}\right)^{p_i}.
    \]
    Rearranging the terms, we group the powers of $b$ and $b-1$:
    \[
    A_n = \frac{1}{b} \sum_{t\ge 1} \frac{1}{\binom{t+b-1}{b}} \sum_{\substack{p_1>\cdots>p_t\ge 1\\p_1+\cdots+p_t=n}} \left(\frac{b-1}{b}\right)^n \prod_{i=1}^t \left(\frac{i+b-1}{i+b-2}\right)^{p_i}
    \]
    Since $\sum_{i=1}^t p_i = n$, we can distribute the factor $\left(\frac{b-1}{b}\right)^n$ into the product:
    \[
    \left(\frac{b-1}{b}\right)^n \prod_{i=1}^t \left(\frac{i+b-1}{i+b-2}\right)^{p_i} = \prod_{i=1}^t \left( \frac{b-1}{b} \cdot \frac{i+b-1}{i+b-2} \right)^{p_i}
    \]
    Now define $\rho_i = \frac{b-1}{b} \cdot \frac{i+b-1}{i+b-2}$.
    For $i=1$, we have $\rho_1 = \frac{b-1}{b} \cdot \frac{1+b-1}{1+b-2} = \frac{b-1}{b} \cdot \frac{b}{b-1} = 1$.
    For $i \ge 2$, we can write
    \[
    \rho_i = \frac{bi + b^2 - b - i - b + 1}{bi + b^2 - 2b} = \frac{bi + b^2 - 2b - (i-1)}{bi + b^2 - 2b} = 1 - \frac{i-1}{b(i+b-2)}
    \]
    Assuming $b > 1$ and $i \ge 2$, we have $0 < \rho_i < 1$.

    Now, we change variables from the partition parts $p_i$ to the differences $\delta_i$. Let $\delta_t = p_t$, and $\delta_{i-1} = p_{i-1} - p_i$ for $i = 1, 2, \dots, t-1$.
    Since $p_1 > p_2 > \dots > p_t \ge 1$, we have $\delta_i \ge 1$ for all $i=1, \dots, t$.
    We can express $p_i$ in terms of $\delta_j$:
    \[
    p_i = \sum_{j=i}^t \delta_j
    \]
    The condition $\sum_{i=1}^t p_i = n$ becomes:
    \[
    \sum_{i=1}^t \sum_{j=i}^t \delta_j = \sum_{j=1}^t j \delta_j = n
    \]
    The product term transforms as follows:
    \[
    \prod_{i=1}^t \rho_i^{p_i} = \prod_{i=1}^t \rho_i^{\sum_{j=i}^t \delta_j} = \prod_{j=1}^t \left( \prod_{i=1}^j \rho_i \right)^{\delta_j}.
    \]
    Let $R_j = \prod_{i=1}^j \rho_i$. Then the term is $\prod_{j=1}^t R_j^{\delta_j}$.
    Since $\rho_1 = 1$, we have $R_1 = 1$. Thus $R_1^{\delta_1} = 1$.
    For $j \ge 2$, $R_j = R_{j-1} \rho_j < R_{j-1}$. Since $R_1=1$, we have $R_j < 1$ for all $j \ge 2$.
    Explicitly,
    \[
    R_j = \prod_{i=1}^j \frac{b-1}{b} \frac{i+b-1}{i+b-2} = \left(\frac{b-1}{b}\right)^j \frac{j+b-1}{b-1} = \frac{(b-1)^{j-1}(j+b-1)}{b^j}
    \]
    Let $C_t = \frac{1}{b \binom{t+b-1}{b}}$. We can rewrite $A_n$ as:
    \[
    A_n = \sum_{t \ge 1} C_t \sum_{\substack{\delta_1, \dots, \delta_t \ge 1 \\ \sum_{j=1}^t j \delta_j = n}} \prod_{j=2}^t R_j^{\delta_j}.
    \]
    Let $S_{n,t}$ be the inner term for a fixed $t$:
    \[
    S_{n,t} = C_t \sum_{\substack{\delta_1, \dots, \delta_t \ge 1 \\ \delta_1 + \sum_{j=2}^t j \delta_j = n}} \prod_{j=2}^t R_j^{\delta_j}.
    \]
    The constraint determines $\delta_1$ because $\delta_1 = n - \sum_{j=2}^t j \delta_j$. The condition $\delta_1 \ge 1$ implies $\sum_{j=2}^t j \delta_j \le n-1$.
    Thus,
    \[
    S_{n,t} = C_t \sum_{\substack{\delta_2, \dots, \delta_t \ge 1 \\ \sum_{j=2}^t j \delta_j \le n-1}} \prod_{j=2}^t R_j^{\delta_j}
    \]
    Note that if $n$ is small such that no such $\delta_j$ exist, the sum is empty and $S_{n,t} = 0$.
    Hence, we define the limit term $L_t$ by removing the upper bound constraint on the sum:
    \[
    L_t = C_t \sum_{\delta_2, \dots, \delta_t \ge 1} \prod_{j=2}^t R_j^{\delta_j} = C_t \prod_{j=2}^t \left( \sum_{\delta=1}^\infty R_j^\delta \right)
    \]
    Since $0 < R_j < 1$ for $j \ge 2$, the geometric series converge:
    \[
    \sum_{\delta=1}^\infty R_j^\delta = \frac{R_j}{1-R_j}.
    \]
    So,
    \[
    L_t = C_t \prod_{j=2}^t \frac{R_j}{1-R_j}.
    \]
    Clearly, for any fixed $t$, as $n \to \infty$, the condition $\sum_{j=2}^t j \delta_j \le n-1$ is eventually satisfied for any fixed tuple $(\delta_2, \dots, \delta_t)$. Since the terms are positive, $S_{n,t}$ is a partial sum of the convergent series defining $L_t$. Thus, $\lim_{n \to \infty} S_{n,t} = L_t$.
    Moreover, $S_{n,t} \le L_t$ for all $n$.

    We now verify that $\sum_{t \ge 1} L_t$ converges.
    We compute the ratio $\frac{R_j}{1-R_j}$:
    \[
    1 - R_j = 1 - \frac{(b-1)^{j-1}(j+b-1)}{b^j} = \frac{b^j - (b-1)^{j-1}(j+b-1)}{b^j}
    \]
    Thus,
    \[
    \frac{R_j}{1-R_j} = \frac{(b-1)^j (j+b-1)}{(b-1)b^j - (b-1)^j(b+j-1)}
    \]
    Substituting this into $L_t$:
    \begin{align*}
        L_t &= C_t \prod_{j=2}^t \frac{(b-1)^j (j+b-1)}{(b-1)b^j - (b-1)^j(b+j-1)} \\
        &= C_t \frac{(b-1)^{\sum_{j=2}^t j} \prod_{j=2}^t (j+b-1)}{\prod_{j=2}^t (b-1)b^j - (b-1)^j(b+j-1)}
    \end{align*}
    We have $\sum_{j=2}^t j = \frac{t(t+1)}{2} - 1 = \frac{(t+2)(t-1)}{2}$.
    Also, $\prod_{j=2}^t (j+b-1) = \frac{(t+b-1)!}{(b+1)!} \cdot (b+1) = \frac{(t+b-1)!}{b!}$.
    We can expand $C_t = \frac{1}{b \binom{t+b-1}{b}} = \frac{b! (t-1)!}{b (t+b-1)!}$.
    Thus,
    \begin{align*}
        L_t &= \frac{b! (t-1)!}{b (t+b-1)!} \cdot \frac{(b-1)^{(t+2)(t-1)/2} (t+b-1)!}{b! \prod_{j=2}^t (b-1)b^j - (b-1)^j(b+j-1)} \\
        &= \frac{1}{b} \frac{(t-1)! (b-1)^{(t+2)(t-1)/2}}{\prod_{j=2}^t (b-1)b^j - (b-1)^j(b+j-1)}
    \end{align*}
    This is exactly the $t$-th term of the series $\gamma_b$ given in the problem.
    For large $t$, $(b-1)b^j - (b-1)^j(b+j-1) \approx (b-1)b^j$, so $\prod (b-1)b^j - (b-1)^j(b+j-1) \approx (b-1)^{t-1} b^{t^2/2}$. The numerator grows like $(b-1)^{t^2/2}$. The ratio behaves like $(b-1/b)^{t^2/2}$, which decays rapidly and thus $\sum L_t$ converges.

    By Tannery's Theorem, since $S_{n,t} \to L_t$ as $n \to \infty$, $|S_{n,t}| \le L_t$, and $\sum L_t < \infty$, we have:
    \[
    \lim_{n \to \infty} A_n = \lim_{n \to \infty} \sum_{t \ge 1} S_{n,t} = \sum_{t \ge 1} \lim_{n \to \infty} S_{n,t} = \sum_{t \ge 1} L_t.
    \]
    Since $\sum_{t \ge 1} L_t = \gamma_b$, we conclude that
    \[
    \lim_{n \to \infty} \frac{S_n(b)}{b^n} = \gamma_b
    \]
    And the inequalities imply $S_n(b) = (\gamma_b - o(1)) b^n$.
\end{proof}

This proves that the new formula provides a significant improvement over the previous bound of $\frac{1}{b} b^n$, as shown in \cref{tab:bound_comparison}.

\begin{table}[H]
    \centering
    \[
    \begin{array}{c|l|l}
         & 1/b & \gamma_b  \\ \hline
        b = 3 & \texttt{0.3333\ldots} & \texttt{~~2.7043\ldots} \\
        b = 4 & \texttt{0.2500\ldots} & \texttt{~~6.9306\ldots} \\
        b = 5 & \texttt{0.2000\ldots} & \texttt{~20.4346\ldots} \\
        b = 6 & \texttt{0.1666\ldots} & \texttt{~65.9828\ldots} \\
        b = 7 & \texttt{0.1428\ldots} & \texttt{226.7906\ldots} \\
    \end{array}
    \]
    \caption{Comparing old and new bound}
    \label{tab:bound_comparison}
\end{table}

\section{Infinite State Graph} \label{sec:infinite-state-graph}

Thus far, we have been considering juggling patterns with a finite number of balls.  We say that the states in every normal juggling pattern with $b$ balls lies in the infinite state graph $G_b$.  Since any state graph is by definition infinite, we will take the convention of dropping the ``infinite'' and calling it a state graph.

\begin{example} \label{ex:inf-sg-G2-juggling-ex}
    Consider the following states, which are contained in the state graph $G_2$:  $\langle 1, 0, 1 \rangle$, $\langle 0, 1, 0, 0, 0, 1 \rangle$, $\langle 1, 0, 0, 0, 1 \rangle$, and $\langle 0, 1, 0, 1 \rangle$.  Together, they form the valid juggling pattern
    \begin{equation*}
        \langle 1, 0, 1 \rangle \rightarrow \langle 0, 1, 0, 0, 0, 1 \rangle \rightarrow \langle 1, 0, 0, 0, 1 \rangle \rightarrow \langle 0, 1, 0, 1 \rangle \rightarrow \langle 1, 0, 1 \rangle
    \end{equation*}
\end{example}

Suppose we want to look a similar pattern but with 3 balls instead.  We can create this pattern by adding a 1 to the beginning of each state.  This gives us the valid juggling pattern
\begin{equation*}
    \langle 1, 1, 0, 1 \rangle \rightarrow \langle 1, 0, 1, 0, 0, 0, 1 \rangle \rightarrow \langle 1, 1, 0, 0, 0, 1 \rangle \rightarrow \langle 1, 0, 1, 0, 1 \rangle \rightarrow \langle 1, 1, 0, 1 \rangle
\end{equation*}
Again suppose we want to look a similar pattern but with 4 balls.  Following the same procedure as above, we have the valid juggling pattern
\begin{equation*}
    \langle 1, 1, 1, 0, 1 \rangle \rightarrow \langle 1, 1, 0, 1, 0, 0, 0, 1 \rangle \rightarrow \langle 1, 1, 1, 0, 0, 0, 1 \rangle \rightarrow \langle 1, 1, 0, 1, 0, 1 \rangle \rightarrow \langle 1, 1, 1, 0, 1 \rangle
\end{equation*}

Since we can repeat this process for $b$  balls, we see that the state graphs are embedded in each other. In other words, as we add balls, we gain structure.  That is, $G_1 \subseteq G_2  \subseteq G_3 \subseteq \cdots \subseteq G_b$.  Not only can we repeat this process for finitely many balls, we can also repeat the process for infinitely many balls.  In such a case, we have the infinite state graph $G_\infty$.  Moreover, we have that $G_1  \subseteq G_2 \subseteq G_3 \subseteq \cdots \subseteq G_\infty$.

When dealing with states in $G_\infty$, we will adopt a variation on how we notate states.  Instead of using angled bracket notation, we will write states as an infinite binary string.  
\begin{example} \label{ex:inf-sg-Ginf-juggling-ex}
    Consider the juggling pattern in \cref{ex:inf-sg-G2-juggling-ex}.  The analogous juggling pattern in $G_\infty$ is as follows:
    \begin{equation*}
        \dots 110100 \ldots \rightarrow \dots 1101000100 \ldots \rightarrow \dots 11000100 \ldots \rightarrow \dots 11010100 \ldots \rightarrow \dots 110100 \dots
    \end{equation*}
\end{example}

However, this notation contains superfluous information, as every state will have a prefix of an infinite number of 1's and a suffix of an infinite number of 0's.  Thus, we will truncate the states, removing both the prefix and suffix.  Since every state also has a 0 immediately following the prefix, we will take the convention of removing that 0 as well.  Thus, we abbreviate the states to contain only the relevant information on the state.  In \cref{ex:inf-sg-Ginf-juggling-ex}, in place of the expanded states, we can now write the pattern with the abbreviated states as such:
\begin{equation*}
    1 \rightarrow 10001 \rightarrow 001 \rightarrow 101 \rightarrow 1
\end{equation*}

Another reason we choose to write states in this abbreviated form is we no longer think of state transitions explicitly as the time until a ball lands.  While this intuition works with a finite number of balls, it breaks down with an infinite number of balls.  Instead, we introduce the following two rules for transitions:
\begin{enumerate}
    \item \textit{Replace any 0 with a 1:}
    This is equivalent to the transition:
    \[
    \dots 1110 s 0 r 000\dots \to \dots 1110 s1r 000\dots
    \]
    where $s$ and $r$ are some finite strings.
    \item \textit{Delete up to and including the first 0 (in the abbreviated state):}
    This is equivalent to the transition:
    \[
    \dots 1110\underbrace{1\dots1}_k0s 000\dots \to \dots 1111\underbrace{1\dots1}_k0s 000\dots
    \]
    where $k$ is some non-negative integer, and $s$ is some finite string.
\end{enumerate}
We leave it to the reader to see that these rules apply in the example above.  We also provide some additional examples, as follows:
\begin{align*}
    001001 &\rightarrow 00100101 \\
    11000101 &\rightarrow 11100101 \\
    001001 &\rightarrow 01001 \\
    11000101 &\rightarrow 00101
\end{align*}

Finally, just as we rewrote a juggling pattern in $G_b$ to one in $G_\infty$, we can take a pattern in $G_\infty$ and write it in $G_b$ for some $b$.  To do so, we find the maximum number $b'$ of 1's for each state in the pattern.  Then, for any $b \geq b'$, we can rewrite the pattern in $G_b$.  For example, from \cref{ex:inf-sg-Ginf-juggling-ex} above, the first time the pattern $1 \rightarrow 10001 \rightarrow 001 \rightarrow 101 \rightarrow 1$ occurs is in $G_2$ as $101 \rightarrow 010001 \rightarrow 10001 \rightarrow 0101 \rightarrow 101$. Notice that  to obtain the pattern in $G_2$, we added to the beginning the difference in 1's followed immediately by a padding 0.

\subsection{2-Ball Base State} \label{sec:base-state2}

We count the number of prime 2-ball juggling patterns containing the state $\langle1,1\rangle$.

\begin{theorem} \label{thm:base-state2-count-11}
    \[
    B(n) = \sum_t \sum_{\substack{p_1 > \ldots > p_t \geq 1 \\ p_1 + \ldots + p_t = n - t}}
    \frac1{t+1} \prod_{i=1}^t \left(\frac{i + 1}{i}\right)^{p_i}
    + \sum_t \sum_{\substack{p_1 + \ldots > p_t \geq 1 \\ p_1 + \ldots + p_t = n - t - 1}}
    \frac1{t+1} \prod_{i=1}^t \left(\frac{i + 1}{i}\right)^{p_i}
    \]
\end{theorem}

\begin{proof}
    Let us count the number of prime cycles of length $n$.
    
    In the typical fashion, we partition $n$ into the maximal spacings within the period.
    
    Since we must include $\langle 1; 1 \rangle$ as a state (we refer to it as the ``base state''), there must be a state that includes a spacing of 1. This can either be the maximal spacing of its set or a non-maximal spacing.
    
    In the first case, suppose that it is the maximal spacing. Then, the partition of $n$ must have minimal element 2, so, with the same method of counting the lesser set elements as previously, we count:
    \[\sum_t \sum_{\substack{p_1 + \ldots > p_t \geq 2 \\ p_1 + \ldots + p_t = n}} \frac1{t + 1} \prod_{i = 1}^t \left( \frac{i + 1}i \right)^{p_i} \]
    The bounds on the sum can be adjusted by taking for granted that at least one element has already been added to each summand of the partition, and then partitioning the remaining elements of $n$. This is equivalent to partitioning $n - t$, as there are $t$ sets.
    \[\sum_t \sum_{\substack{p_1 + \ldots > p_t \geq 1 \\ p_1 + \ldots + p' = n - t}} \frac{1}{t + 1} \prod_{i = 1}^t \left( \frac{i + 1}{i} \right)^{p_i} \]
    
    In the second case, in which 1 is a lesser spacing, we still may not include any sets with a maximum spacing 1. The length is partitioned without sets of maximum spacing 1, and without the throw that results in the base state, which is added in at the end. This yields the following count:
    \[\sum_t \sum_{\substack{p_1 + \ldots > p_t \geq 1 \\ p_1 + \ldots + p_t = n - t - 1}} \frac1{t + 1} \prod_{i = 1}^t \left( \frac{i + 1}i \right)^{p_i} \]
    
The sum of these sums covers all the cases in which the base state could occur within a prime 2-ball pattern.
\end{proof}

From this theorem, we can express the count $B(n)$ more compactly using the function $c_t(n)$. Note that the inner sums in \cref{thm:base-state2-count-11} differ from $c_t(m)$ only by a factor of $t$.

\begin{corollary}\label{cor:base-state-formula}
The number of 2-ball prime patterns containing the base state $\langle 1, 1 \rangle$ is given by
\[
B(n) = \sum_{t=1}^\infty t \cdot c_t(n-t) + \sum_{t=1}^\infty t \cdot c_t(n-t-1)
\]
\end{corollary}

\subsection{Asymptotics}

We now determine the asymptotic behavior of $B(n)$ by applying the General Asymptotic Theorem (\cref{thm:general_asymptotics}).

We observe that for large $n$, the shifted term $c_t(n-k)$ behaves asymptotically as a scaled version of $c_t(n)$. Specifically, since $c_t(n) \sim q_t 2^n$, we have $c_t(n-k) \sim q_t 2^{n-k} = 2^{-k} (q_t 2^n) \sim 2^{-k} c_t(n)$.

This allows us to interpret the sums in \cref{cor:base-state-formula} as weighted sums of the form $F_w(n) = \sum w(t) c_t(n)$ with modified weights.
\begin{itemize}
    \item For the first term, $\sum t \cdot c_t(n-t)$, the weight is $w_1(t) = t \cdot 2^{-t}$.
    \item For the second term, $\sum t \cdot c_t(n-t-1)$, the weight is $w_2(t) = t \cdot 2^{-(t+1)}$.
\end{itemize}

We now compute the asymptotic constants $\gamma_{w_1}$ and $\gamma_{w_2}$ as defined in \cref{thm:general_asymptotics}.
\begin{align*}
\gamma_{w_1} &= \sum_{t=1}^\infty w_1(t) q_t = \sum_{t=1}^\infty t 2^{-t} q_t \\
\gamma_{w_2} &= \sum_{t=1}^\infty w_2(t) q_t = \sum_{t=1}^\infty t 2^{-(t+1)} q_t = \frac{1}{2} \sum_{t=1}^\infty t 2^{-t} q_t = \frac{1}{2} \gamma_{w_1}
\end{align*}
The total asymptotic constant is $\gamma_{B} = \gamma_{w_1} + \gamma_{w_2} = \frac{3}{2} \gamma_{w_1}$.

\begin{theorem}
The number of 2-ball prime cycles containing the base state $\langle 1, 1 \rangle$ satisfies
\[
B(n) \sim \gamma_{B} 2^n
\]
where
\[
\gamma_{B} = \frac{3}{2} \sum_{t=1}^\infty \frac{t q_t}{2^t} = \frac{3}{4} \sum_{t=1}^\infty t \left( \frac{1}{2^t} \prod_{i=2}^t \frac{i-1}{2^i - i - 1} \right).
\]
\end{theorem}

\begin{proof}
    Applying \cref{thm:general_asymptotics} to the weighted sums described above, we obtain:
    \begin{align*}
        B(n) &\sim \gamma_{w_1} 2^n + \gamma_{w_2} 2^n \\
        &= \left( \sum_{t=1}^\infty \frac{t q_t}{2^t} + \frac{1}{2} \sum_{t=1}^\infty \frac{t q_t}{2^t} \right) 2^n \\
        &= \frac{3}{2} \left( \sum_{t=1}^\infty \frac{t q_t}{2^t} \right) 2^n
    \end{align*}
\end{proof}

\subsection{Fixed State} \label{sec:fixed-state}

Define a flip-reverse function $\FR(\alpha)$ that takes in a state $\alpha \in G_\infty$ and returns another state $\alpha' \in G_\infty$.  To perform the function, we look at $\alpha$ as an infinite state.  Then, we flip all digits so that 0's become 1's, and vice versa.  Finally, we reverse the order of the digits in the state.  An example is as follows: 
\begin{example}
    Consider the state $\alpha = 0011$.
    \begin{align*}
        \alpha = 0011 & \Rightarrow \ldots 110001100 \ldots && \text{ (expand state)} \\
        & \Rightarrow \ldots 001110011 \ldots && \text{ (flip all digits)} \\
        & \Rightarrow \ldots 110011100 \ldots && \text{ (reverse)} \\
        & \Rightarrow 0111 = \alpha' && \text{ (abbreviate state)} 
    \end{align*}
\end{example}

Consider two states $\alpha$ and $\beta$.  We say that $\FR(\alpha) = \beta$ if after performing the flip-reverse function on $\alpha$, we have the state $\beta$.  Also, notice that if we perform the flip-reverse function on $\FR(\alpha)$, then we have $\alpha$; that is, $\FR(\FR(\alpha)) = \alpha$.  Therefore, the flip-reverse function is an involution.

\begin{lemma} \label{lem:fr-states-bijection}
    The flip-reverse function $\FR(\alpha)$ is bijective between states of the infinite state graph $G_\infty$.
\end{lemma}

\begin{proof}
    The flip-reverse function is an involution.  By definition, it is also a bijection.
    %
\end{proof}

\begin{lemma} \label{lem:fr-transitions-automorphism}
    If $\alpha \rightarrow \beta $ is a valid transition in $G_\infty$, then $\FR(\beta) \rightarrow \FR(\alpha)$ is a valid transition in $G_\infty$.
\end{lemma}

First, we provide some intuition for \cref{lem:fr-transitions-automorphism}.  Consider a transition $\alpha \rightarrow \beta$ in $G_\infty$.  Recall the two rules for transitions:  replace a 0 with a 1, or delete up to and including the first 0 in the abbreviated state.  However, if we consider expanded states, the second rule is exactly the first, but in this instance, we replace the padding 0.  We can see this more clearly in the following example:
\begin{align*}
    11000101 \hspace{27.8pt} & \rightarrow \hspace{51pt} 00101 \\
    \Updownarrow \hspace{50pt} & \hspace{75pt} \Updownarrow \\
    \ldots 11{\textcolor{red}{0}}1100010100 \ldots &\rightarrow \ldots 11{\textcolor{red}{1}}1100010100 \ldots
\end{align*}

Then, given that transitions are formed replacing any one 0 with a 1, then it must be the case that $\beta$ contains one ``more'' 1 than $\alpha$.  It also must be the case that $\FR(\beta)$ contains one ``more'' 0 than $\FR(\alpha)$.  Therefore, the automorphic transition to $\alpha \rightarrow \beta$ is $\FR(\beta) \rightarrow \FR(\alpha)$.

\begin{proof}[Proof of \cref{lem:fr-transitions-automorphism}]
    Consider a transition $\alpha \rightarrow \beta$ in $G_\infty$.  By the transition rules, it must be the case that exactly one 0 in $\alpha$ was replaced with a 1 in $\beta$.  Now, consider $\FR(\alpha)$ and $\FR(\beta)$.  The only difference between the two states is exactly one digit.  Furthermore, that digit must be a 1 in $\FR(\alpha)$ that was replaced with a 0 in $\FR(\beta)$.  Therefore, $\FR(\beta) \rightarrow \FR(\alpha)$ is a valid transition.
    %
\end{proof}

We now give an example of \cref{lem:fr-transitions-automorphism}.  Consider the state $\alpha = 01$, which transitions to the state $\beta = 1$.  Using the flip-reverse function, we have the states $\alpha ' = \FR(\alpha) = 11$ and $\beta ' = \FR(b) = 1$.  It is clear that $\beta ' \rightarrow \alpha '$ is a valid transition.

\begin{theorem} \label{thm:fr-walks-bijection}
    Consider two states $\alpha$ and $\beta = \FR(\alpha)$ in the infinite state graph $G_\infty$.  Let $A$ be the set containing all $k$-length walks that pass through $\alpha$, and let $B$ be the set containing all $k$-length walks that pass through $\beta$.  Then, there exists a bijection between the walks in set $A$ and the walks in set $B$.
\end{theorem}

\begin{proof}
    Consider a walk $P_\alpha$, where $P_\alpha \in A$.  By \cref{lem:fr-states-bijection}, for each state $\alpha _i$ in $P_\alpha$, there exists a corresponding state $\beta _i = \FR(\alpha_i)$, which is unique.  By \cref{lem:fr-transitions-automorphism}, for each transition $\alpha _i \rightarrow \alpha _{i + 1}$ in $P_\alpha$, there exists a corresponding transition $\beta _{i + 1} \rightarrow \beta _i$.  Therefore, it must be the case that for each $P_\alpha \in A$, there exists a corresponding walk $P_\beta \in B$ that is unique.  Therefore, there exists an injection between sets $A$ and $B$.  
    
    Since this process can be repeated for any walk in $B$, there exists a surjection between the sets.  Since there exist both an injection and a surjection between the two sets, there exists a bijection between set $A$ and set $B$.
    
\end{proof}

%

\section*{Acknowledgments}
This research was conducted at the mathematics 2024 REU program held at Iowa State University, which was supported by the National Science Foundation under Grant No. DMS-1950583.

\printbibliography[title={References}]

\end{document}

%% file: paper-pics.tex
\def\introstategraph{
\begin{tikzpicture}
\node at (0,0) {$\langle0,1,0,1\rangle$};
\node at (3,0) {$\langle1,0,1,0\rangle$};
\node at (6,0) {$\langle1,1,0,0\rangle$};
\node at (10,0) {$\langle0,0,1,1\rangle$};
\node at (8,1.5) {$\langle1,0,0,1\rangle$};
\node at (8,-1.5) {$\langle0,1,1,0\rangle$};
\draw[->, thick] (1,0.125)--(2,0.125);
\draw[->, thick] (2,-0.125)--(1,-0.125);
\draw[->, thick] (4,0.125)--(5,0.125);
\draw[->, thick] (5,-0.125)--(4,-0.125);
\draw[->,thick,rounded corners] (7,-0.125)--(7.25,-0.125)--(7.25,0.125)--(7,0.125);
\draw[->,thick] (8,1.125)--(8,-1.125);
\draw[->,thick] (6.5,0.375)--(7.5,1.125);
\draw[->,thick] (7.5,-1.125)--(6.5,-0.375);
\draw[->,thick] (8.5,1.125)--(9.5,0.375);
\draw[->,thick] (9.5,-0.375)--(8.5,-1.125);
\draw[->,thick] (7,1.5)--(3.5,0.375);
\draw[->,thick] (3.5,-0.375)--(7,-1.5);
\node at (1.5,0.35) {\small $0$};
\node at (1.5,-0.35) {\small $4$};
\node at (4.5,0.35) {\small $1$};
\node at (4.5,-0.35) {\small $3$};
\node at (7.45,0) {\small $2$};
\node at (8.2,0) {\small $2$};
\node at (6.85,0.95) {\small $4$};
\node at (6.85,-0.95) {\small $0$};
\node at (9.15,0.95) {\small $4$};
\node at (9.15,-0.95) {\small $0$};
\node at (5.2,1.15) {\small $1$};
\node at (5.2,-1.15) {\small $3$};
\end{tikzpicture}}

\def\cardA#1#2{
\begin{scope}[shift={(#1,#2)}]
    \begin{scope}
        \clip (0,0) rectangle (1,1.33);
        \draw[ultra thick, color=red!50!black] (0,1)--(1,1);
        \draw[ultra thick, color=red!50!black] (0,.66)--(1,.66);
    \end{scope}
    \draw[thick,rounded corners] (0,0) rectangle (1,1.33);
\end{scope}
}

\def\cardB#1#2{
\begin{scope}[shift={(#1,#2)}]
    \begin{scope}
        \clip (0,0) rectangle (1,1.33);
        \draw[ultra thick, color=red!50!black] (0,1)--(1,1);
        \draw[ultra thick, color=red!50!black] (0,.66) to[out=0,in=90] (0.5,0) to[out=90,in=180] (1,.66);
    \end{scope}
    \draw[thick,rounded corners] (0,0) rectangle (1,1.33);
\end{scope}
}

\def\cardC#1#2{
\begin{scope}[shift={(#1,#2)}]
    \begin{scope}
        \clip (0,0) rectangle (1,1.33);
        \draw[ultra thick, color=red!50!black] (0,1) to[out=0,in=180] (1,0.66);
        \draw[line width=4pt, color=white] (0,.66) to[out=0,in=90] (0.5,0) to[out=90,in=180] (1,1);
        \draw[ultra thick, color=red!50!black] (0,.66) to[out=0,in=90] (0.5,0) to[out=90,in=180] (1,1);
    \end{scope}
    \draw[thick,rounded corners] (0,0) rectangle (1,1.33);
\end{scope}
}
\def\cardD#1#2{
\begin{scope}[shift={(#1,#2)}]
    \begin{scope}
        \clip (0,0) rectangle (1,1.33);
        \draw[ultra thick, color=red!50!black] (0,1) to[out=0,in=180] (1,0.83);
        \draw[ultra thick, color=red!50!black] (0,.66) to[out=0,in=90] (0.5,0);
        \draw[ultra thick, color=red!50!black] (0.5,0) to[out=90,in=180] (1,0.83);
    \end{scope}
    \draw[thick,rounded corners] (0,0) rectangle (1,1.33);
\end{scope}
}

\def\cardE#1#2{
\begin{scope}[shift={(#1,#2)}]
    \begin{scope}
        \clip (0,0) rectangle (1,1.33);
        \draw[line width=3pt, color=red!50!black] (0,.83)--(1,0.83);
    \end{scope}
    \draw[thick,rounded corners] (0,0) rectangle (1,1.33);
\end{scope}
}

\def\cardF#1#2{
\begin{scope}[shift={(#1,#2)}]
    \begin{scope}
        \clip (0,0) rectangle (1,1.33);
        \draw[line width=3pt, color=red!50!black] (0,.83) to[out=0,in=90] (0.5,0);
        \draw[ultra thick, color=red!50!black] (0.5,0) to[out=90,in=180] (1,.66);
        \draw[ultra thick, color=red!50!black] (0.5,0) to[out=90,in=180] (1,1);
    \end{scope}
    \draw[thick,rounded corners] (0,0) rectangle (1,1.33);
\end{scope}
}

\def\cardG#1#2{
\begin{scope}[shift={(#1,#2)}]
    \begin{scope}
        \clip (0,0) rectangle (1,1.33);
        \draw[line width=3pt, color=red!50!black] (0,.83) to[out=0,in=90] (0.5,0);
        \draw[line width=3pt, color=red!50!black] (1,.83) to[out=180,in=90] (0.5,0);
    \end{scope}
    \draw[thick,rounded corners] (0,0) rectangle (1,1.33);
\end{scope}
}

\def\ucycle#1#2#3#4#5{
\begin{scope}[shift={(#1,#2)}]
    \def\n{#3}
    \def\r{#4}
    \pgfmathtruncatemacro{\shift}{#5}

    \foreach \i in {1,...,\n} {
        \pgfmathtruncatemacro{\position}{mod(\i + \shift - 1, \n) + 1}
        \node (S\i) at ({90 - 360 / \n * (\position - 1)}:\r) {$S_{\i}$};
    }
\end{scope}
}

\def\unode#1#2#3#4#5{
\begin{scope}[shift={(#1,#2)}]
    \def\n{#3}
    \def\r{#4}
    \pgfmathtruncatemacro{\shift}{mod(#5, \n) + 1}

    \draw[violet, thick] ({90 - 360 / \n * (\shift - 1)}:\r) circle (0.37);
\end{scope}
}

\def\dnode#1#2#3#4#5{ 
\begin{scope}[shift={(#1,#2)}]
    \def\n{#3}
    \def\r{#4}
    \pgfmathtruncatemacro{\shift}{mod(#5, \n) + 1}
    \pgfmathsetmacro{\halfangle}{180 / \n}

    \draw[green, thick] ({90 - 360 / \n * \shift + \halfangle}:\r) circle (0.37);
\end{scope}
}

\def\ndnode#1#2#3#4#5{ 
\begin{scope}[shift={(#1,#2)}]
    \def\n{#3}
    \def\r{#4}
    \pgfmathtruncatemacro{\shift}{mod(#5, \n) + 1}
    \pgfmathsetmacro{\halfangle}{180 / \n}
    \pgfmathsetmacro{\x}{90 - 360 / \n * \shift + \halfangle}
    \pgfmathsetmacro{\xpos}{\r * cos(\x)}
    \pgfmathsetmacro{\ypos}{\r * sin(\x)}

    \def\crosssize{0.25}

    \draw[red, thick] (\xpos-\crosssize, \ypos-\crosssize) -- (\xpos+\crosssize, \ypos+\crosssize);
    \draw[red, thick] (\xpos-\crosssize, \ypos+\crosssize) -- (\xpos+\crosssize, \ypos-\crosssize);
\end{scope}
}

\def\multiplexexample{
    \begin{scope}
        \pgfmathsetmacro{\width}{8};
        \pgfmathsetmacro{\height}{5};
        
        \foreach \x in {0,...,\width}
            \draw[color=black!40!white, dotted, thick] (\x,0)--(\x,\height);
        \foreach \y in {0,...,\height}
            \draw[color=black!40!white, dotted, thick] (-0.5,\y)--(\width+0.5,\y);
        
        \draw[thick,color=black!100!white] (3,2)--(3,5)  (5,0)--(5,2)  (7,0)--(7,5)  (3,5)--(8,0)  (3,2)--(5,0)  (5,2)--(7,0)  (7,5)--(8,4)  (8,0)--(8,4);
        
        \draw[line width = 0.8pt, color=black!100!white,double] (-0.5,3.5)--(3,0)  (3,0)--(3,2)  (8,4)--(8.5,3.5);
        
        \foreach \pair in {(0,3),(1,2),(2,1),(8,4)}
            \draw[fill=black!20!white,thick] \pair circle (0.1) \pair circle (0.175) ;
        
        \foreach \pair in {(3,5),(3,2),(4,4),(4,1),(5,3),(5,2),(6,2),(6,1),(7,5),(7,1)}
            \draw[fill=black!20!white,thick] \pair circle (0.125);
        
    \end{scope}
}

\def\coloredexample{
    \begin{scope}
        \pgfmathsetmacro{\width}{9};
        \pgfmathsetmacro{\height}{5};
        
        \foreach \x in {0,...,\width}
            \draw[color=black!40!white, dotted, thick] (\x-1,0)--(\x-1,\height);
        \foreach \y in {0,...,\height}
            \draw[color=black!40!white, dotted, thick] (-0.5-1,\y)--(\width+0.5-1,\y);
        
        \draw[line width = 1.4 pt,color=gray!20!orange] (-1.5,1.5)--(0,0)--(0,4)--(4,0)--(4,1)--(5,0)--(5,2)--(7,0)--(7,3)--(8.5,1.5);
        
        \draw[line width = 1.4 pt,color=blue!60!white] (-1.5,3.5)--(2,0)--(2,1)--(3,0)--(3,5)--(8,0)--(8,4)--(8.5,3.5);
        
        \foreach \pair in {(0,4),(1,3),(2,2),(3,1),(4,1),(5,2),(6,1),(7,3),(8,2),(-1,1)}
            \draw[fill=gray!20!orange,thick] \pair circle (0.125);
        
        \foreach \pair in {(0,2),(1,1),(2,1),(3,5),(4,4),(5,3),(6,2),(7,1),(8,4),(-1,3)}
            \draw[fill=blue!60!white,thick] \pair circle (0.125);
    \end{scope}
}